\documentclass{amsart}
\usepackage[margin=3cm]{geometry}
\usepackage[dvips]{color}
\usepackage{graphicx}
\usepackage{epsfig}
\usepackage{tikz}
\usepackage{enumerate}
\usepackage{color}
\usepackage[T1]{fontenc}
\usepackage{mathrsfs}

\usepackage{latexsym}
\usepackage{amssymb}
\usepackage{amsmath}
\usepackage{amsthm}
\usepackage{amscd}
\usepackage{subfigure}
\usepackage{hyperref}

\theoremstyle{plain}
\newtheorem{thm}{Theorem}[section]
\newtheorem{lem}[thm]{Lemma}
\newtheorem{prop}[thm]{Proposition}
\newtheorem{cor}[thm]{Corollary}

\newcommand{\R}{\mathbb{R}}
\newcommand{\Z}{\mathbb{Z}}
\newcommand{\N}{\mathbb{N}}
\newcommand{\T}{\mathbb{T}}
\newcommand{\sgn}{\operatorname{sgn}}
\newcommand{\vep}{\varepsilon}
\newcommand{\Leb}{\ensuremath{\lambda}}
\newcommand{\spa}{\operatorname{span}}
\def\cC{\mathcal{C}}

\newcommand{\Log}[1]{\mathcal{L}\textrm{og} \left( #1 \right)}
\newcommand{\SymLog}[1]{\mathcal{S}\textrm{ym}\mathcal{L}\textrm{og} \left( #1 \right)}
\newcommand{\pLog}[1]{\mathcal{L}\textrm{og}^\textrm{p} \left( #1 \right)}
\newcommand{\pSymLog}[1]{\mathcal{S}\textrm{ym}\mathcal{L}\textrm{og}^\textrm{p} \left( #1 \right)}

\theoremstyle{definition}
\newtheorem{rem}{Remark}[section]

\newtheorem{defn}{Definition}[section]

\newcommand{\ep}{\varepsilon}
\def\cC{\mathcal{C}}

\title[Singularity of the spectrum in genus two]{Singularity of the spectrum for smooth area-preserving flows in genus two and translation surfaces well approximated by cylinders}
\author[J.\ Chaika]{Jon Chaika}
\address{Department of Mathematics, University of Utah, 155 South 1400 East, Salt Lake City, Utah, 84112 USA}
\email{chaika@math.utah.edu }

\author[K.\ Fr\k{a}czek]{Krzysztof Fr\k{a}czek}
\address{Faculty of Mathematics and Computer Science, Nicolaus
Copernicus University, ul.\ Chopina 12/18, 87-100 Toru\'n, Poland}
\email{fraczek@mat.umk.pl}

\author[A.\ Kanigowski]{Adam Kanigowski}
\address{Department of Mathematics, University of Maryland, 4176 Campus Drive, William E.\ Kirwan Hall,
College Park, MD 20742-4015, USA}
\email{akanigow@umd.edu}

\author[C.\ Ulcigrai]{Corinna Ulcigrai}
\address{Institut f\"ur Mathematik, Universit\"at Z\"urich, Winterthurerstrasse 190,
CH-8057 Z\"urich, Switzerland}
\email{corinna.ulcigrai@math.uzh.ch}


\begin{document}

\begin{abstract}
We consider smooth flows preserving a smooth invariant measure, or, equivalently, locally Hamiltonian flows on compact orientable surfaces and show that, when the genus of the surface is two, almost every such locally Hamiltonian flow with two non-degenerate isomorphic saddles has singular spectrum.

More in general, singularity of the spectrum holds for special flows over a full measure set of interval exchange transformations with a hyperelliptic permutation (of any number of exchanged intervals), under a roof with symmetric logarithmic singularities. The result is proved using a criterion for singularity based on tightness of Birkhoff sums with exponential tails decay.

A key ingredient in the proof, which is of independent interest, is a result on translation surfaces well approximated by single cylinders. We show that for almost every translation surface in any connected component of any stratum there exists a full measure set of directions which can be well approximated by a single cylinder of area arbitrarily close to one.  The result, in the special case of the stratum $\mathcal{H}(1,1)$, yields rigidity sets needed for the singularity result.
\end{abstract}

\maketitle

\medskip

\begin{flushright}
{\it dedicated to Anatole Katok}
\end{flushright}

\bigskip

This paper provides a first general result on the nature of the spectrum for typical smooth area-preserving flows on surfaces of higher genus. Area-preserving flows are one of the most basic examples of dynamical systems, studied since Poincar{\'e} at the  dawn  of the study of dynamical systems.
We consider the natural class of smooth flows preserving a smooth invariant measure on surfaces of genus $g\geq 1$, also  known as  \emph{locally Hamiltonian flows} (see  \S\ref{sec:locHam}) or equivalently multivalued Hamiltonian flows. The study of locally Hamiltonian flows  has been pushed since the $1990s$ by Novikov and his school for its connection with solid state physics and pseudo-periodic topology (see e.g.~\cite{No:the} and \cite{Zo:how}). Locally Hamiltonian flows arise indeed in the Novikov model of motion of an electron in a metal under a magnetic field - in this semi-classical approximation, the (compact) surface which constrains the motion is then the (quotient of the) periodic Fermi energy level surface of the metal. Basic ergodic properties (such as minimality and ergodicity) of such flows can be deduced\footnote{Locally Hamiltonian flows (when minimal, or restricted to a minimal component) can indeed be seen as \emph{singular time-reparametrizations} of translation flows (see Remark~\ref{rem:reparametrization}) and  properties such as ergodicity and minimality depend only on the flow orbits and not on the time-reparametrization.} from classical results (such as~\cite{Keane,Ma:int, Ve:gau}) on translation flows (which are well understood thanks to the connection with Teichm{\"u}ller dynamics, see e.g.~\cite{AF:wea, Ma:erg}). On the other hand, finer ergodic and spectral properties depend on the nature of the locally Hamiltonian parametrization and on the type of fixed points of the flow.

In the past decades, there have been many advances  in our understanding of finer ergodic properties of locally Hamiltonian flows, in particular mixing and rigidity properties, starting from a conjecture by Arnold on mixing in locally Hamiltonian flows in genus one (see \cite{Ar:top} and \cite{SK:mix}), which led naturally to the study of mixing (and weak mixing) in higher genus locally Hamiltonian flows \cite{Ul:mix, Ul:wea, Sch:abs, Ul:abs, Ra:mix}, up to recent results on mixing of all orders \cite{FK, KKU} and disjointness phenomena \cite{KLU, BK}, some of which were  achieved adapting to the world of smooth flows with singularities tools inspired from homogeneous dynamics and the work of Marina Ratner (a quick review of the known result is presented in Section~\ref{sec:history}).

The {spectral properties (and in particular what is the spectral type}, see \S\ref{sec:spectral} for definitions) of locally Hamiltonian flows is a natural question, which has been lingering for decades (see e.g.~\cite[Section 6]{KT} and \cite{L})\footnote{The result on the singular nature of the spectrum proved in this paper was furthermore explicitly suggested by A.~Katok to A.~Kanigowski~in private communication.}. 
  Results on the spectrum of the operator, though, are very rare. In an early work by Fr\k{a}czek and Lema\'nczyk \cite{Fr-Le0}, spectral properties of  special flows over rotations with single symmetric logarithmic singularity (see \S\ref{sec:reduction}) are examined.  In \cite[Theorem 12]{Fr-Le0} it is shown that (for a full measure set of rotation numbers) such special flows have purely singular continuous spectrum\footnote{The authors in \cite{Fr-Le0} show that special flows over rotations under a symmetric logarithm, for a full measure set of frequencies, are \emph{spectrally disjoint} from all mixing flows (see Theorem~12, from which it follows in particular that the spectrum is purely singular).}. This gives examples of locally Hamiltonian flows on surfaces of any genus $\geq 1$ with  singular continuous spectrum  (see \cite[Theorem 1]{Fr-Le0}).
This result shows that, when one can prove absence of mixing and  some form of (partial) rigidity, it might be possible to deduce singularity of the spectrum. A recent spectral breakthrough, which goes in the opposite direction, was achieved by Fayad, Forni and Kanigowski in \cite{FFK}, who showed that a class of smooth flows on surfaces of genus one (which can also be represented as special flows over rotations, see \S\ref{sec:history}) has countable Lebesgue spectrum. These flows display a strong form of \emph{shearing} of nearby trajectories and
 were proved to be mixing by Kochergin  in the 70's, see \cite{Ko:mix}.

The main result of this paper concerns the nature of  the spectrum of locally Hamiltonian flows on genus two surfaces, and, to the best of our knowledge, is the first general spectral result  for surfaces of higher genus ($g \geq 2)$.

\section{Main results}

We now state the main result on the spectrum of locally Hamiltonian flows on  genus two surfaces (see \S\ref{sec:main_g2}), as well as a result in the language of special flows from which it is deduced, see \S\ref{sec:main_sf}.  The singularity criterion which is used to prove the first two results is stated (and proved) later in the paper, in Section~\ref{sec:singcrit} (as Theorem~\ref{thm:singcrit}). In \S\ref{sec:main_cyl} we state a result on translation surfaces being well approximated by a single cylinder which is used as a key technical tool in the proof, but is also of independent interest, since it concerns Diophantine approximation-type questions for cylinders on translation surfaces in any genus (more precisely, any connected component of any stratum, see \S\ref{sec:main_cyl}).

\subsection{Singularity of the spectrum of locally Hamiltonian flows in genus two}\label{sec:main_g2}
Throughout the paper let $M$ denote a  smooth, compact, connected, orientable surface and  let  $(\varphi_t)$ be  a smooth flow preserving a smooth invariant measure (i.e.\ a measure with smooth positive  density with respect to the area form on $M$). Equivalently,  $(\varphi_t)$ is a  \emph{locally Hamiltonian flow}, see (\S\ref{sec:locHam}). We assume that $(\varphi_t)$ has \emph{non-degenerate} fixed points and is \emph{minimal}. When the surface has genus $g=2$, this implies that there are \emph{two} fixed points, both of which are \emph{simple saddles} (i.e.\ four-pronged saddles, with two incoming and two outgoing separatrices), see Figure~\ref{g2flow}. We will assume furthermore that the two saddles are \emph{isomorphic} (in a sense specified in Section~\ref{sec:singularities}, see Definition~\ref{def:isomorphic}).

 \begin{figure}[h!]
\includegraphics[width=0.6\textwidth]{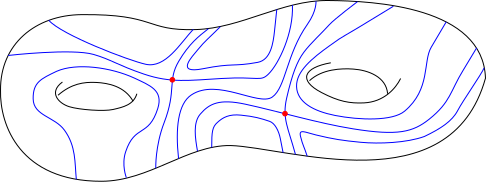}
 \caption{ Trajectories of a locally Hamiltonian flow with two simple saddles on a surface of genus two.\label{g2flow}}
\end{figure}

\begin{thm}[Singular spectrum in genus two]\label{thm:singsp}
A \emph{typical}  locally Hamiltonian flow on a surface $M$ of genus two with two isomorphic saddles has \emph{purely singular} spectrum.
\end{thm}
 Basic spectral notions, and in particular the  definition of \emph{singular spectrum}, are recalled in \S\ref{sec:spectral}. The notion of \emph{typical} used here is in a \emph{measure theoretic sense} and it refers to a full measure set with respect to a natural measure class on locally Hamiltonian flows with given singularity types (sometimes referred to as the Katok fundamental class). 
{For the definition of the measure class and the notion of typical used in the statement of Theorem~\ref{thm:singsp}, see  \S\ref{sec:isomorphic} (and also more in general \S\ref{sec:measureclass}).}

Theorem \ref{thm:singsp} is the first result on singularity of the spectrum of typical minimal locally Hamiltonian flows with non-degenerate singularities  on surfaces in higher genus (and, to the best of our knowledge, the first  general spectral result for smooth flows on surfaces of genus $g\geq 2$). We believe that the result is not only true in genus two, but in any genus $g\geq 2$.  The importance of considering the case of genus two (to deal with some of key difficulties arising when passing from genus one to higher genus, or, in other words, from Poincar{\'e} sections which are rotations to interval exchange transformations), as well as the importance of the assumption  that saddles are isomorphic for the strategy and techniques  of proof will be  explained in  \S\ref{sec:strategy} below.

Singularity of the spectrum is in stark contrast with the recent result in \cite{FFK} on flows on tori with a degenerated singularity (or stopping point), which are shown to have absolutely continuous (and actually countable Lebesgue) spectrum. It might be conjectured, from their result, that also in higher genus, in presence of sufficiently strong \emph{degenerate} singular points, the spectrum is also absolutely continuous (and even countable Lebesgue).
 We remark that stopping points or non-degenerate fixed points (including centers) are known to produce mixing \cite{Ko:mix} (at rates which are expected to be polynomial, see e.g.~\cite{Fa1}), while typical minimal locally Hamiltonian flows with  \emph{non-degenerate} saddles are  known \emph{not} to be mixing  by the work of Scheglov \cite{Sch:abs} for genus two and Ulcigrai \cite{Ul:abs} for any genus. At the heart of our proof is a strengthening of results on absence of mixing (in particular of the works \cite{Fr-Le0,Fr-Le2} and \cite{Sch:abs}). When the flow is not minimal, 
 and has non-degenerate singularities it has 
 \emph{several} minimal components, and the nature of the spectrum (for the restriction of a typical flow to a minimal component) is unclear. These flows are indeed mixing, but with sub-polynomial rate (see \cite{Ra:mix}, which provides logarithmic upper bounds) and it is not clear whether to expect singularity or  absolute continuity of the spectrum.

\subsection{Special flows with symmetric logarithmic singularites over symmetric IETs}\label{sec:main_sf}
Formally, Theorem \ref{thm:singsp}  is deduced from a result for special flows (see below, or \S\ref{sec:sfdef} for formal definitions). It is 
  well known that  any  minimal (or minimal component of) locally Hamiltonian flow can be represented as \emph{the special  flow} over an  \emph{interval exchange transformations} or, for short, IET (see Section~\ref{sec:locHam_sf} for definitions and for the reduction). Our main result, that {certain} special flows have singular spectrum 
holds for IETs on \emph{any} number of intervals   in a special class (corresponding to \emph{symmetric} permutations, or hyperelliptic strata).  Let us give some definitions to formulate the precise statement.

\smallskip
An interval exchange transformation (IET) of $d$ intervals $T:I\to I$ ($I=[0,|I|)$)\footnote{We usually assume that $|I|\leq 1$}  \emph{with permutation} $\pi$ (on $\{0,\dots, d-1\}$)  and \emph{endpoints} (of the continuity intervals)
$0=:\beta_0 < \beta_1 < \dots \beta_{d-1}< \beta_d:= |I|$
is a piecewise isometry which sends the interval $I_i:=[\beta_i,\beta_{i+1})$, for $0\leq i<d$,
by a translation, explicitly given by
\[
T(x) = x-\beta_i +\beta_{\pi(i)}, \qquad \text{if }\ x \in [\beta_i,\beta_{i+1}).
\]
We say that $\pi:\{0,1,\dots, d-1\} \to \{0,1,\dots, d-1\}$ is \emph{symmetric} if $\pi(i) = d-1-i$ for $0\leq i <d$. Thus, in an IET with  a  symmetric permutation the order of the exchanged intervals is \emph{reversed}. These are IETs which arise when considering (suitably chosen Poincar{\'e} sections of translation flows) \emph{hyperelliptic strata} of translation surfaces, in particular for any genus $g \geq 1$  (see for example Lemma~\ref{lemma:reduction}).

We say that a result holds \emph{for almost every IET with permutation} $\pi$ if it holds for almost every choice of the lengths $|I_i|=\beta_{i+1}-\beta_i$ of the exchanged intervals (with respect to the restriction of the Lebesgue measure on $\mathbb{R}^d$ to the simplex $\Delta_{d-1} = \{(\lambda_1,\dots  \lambda_d), \lambda_i\geq 0, \sum_{i=0}^{d-1} \lambda_i=1\}$).

\smallskip
The  \emph{special flow} over $T: I\to I$ under a positive, integrable roof function $f$ (see also \S\ref{sec:sfdef}) is  the vertical, unit speed flow on the region $X_f$  below the graph of $f$, given by $X_f := \{ (x,y) \in I\times\mathbb{R} :   0\leq y < f(x) \}$, with the identification of each point on the graph, of the form $(x,f(x))$, where $x \in I$, with the base point $ (T(x),0)$, as shown in Figure~\ref{symlog} (see \S\ref{sec:sfdef} for formal definitions).



We consider special flows under a \emph{roof function} chosen in a class of (positive) functions which have \emph{logarithmic singularities} at the discontinuities $\beta_i$. This is the type of singularities that arise in the special flow representation of locally Hamiltonian flows with simple saddles, see \S\ref{sec:reduction}. More precisely, the class of functions, denoted by $\SymLog{\sqcup_{i=0}^{d-1}I_i}$ (to refer to \emph{Symmetric Logarithmic} singularities), consists  of positive real valued functions, defined on $\bigcup_{i=0}^d (\beta_i,\beta_{i+1})$ and such that the restriction $f\vert (\beta_i,\beta_{i+1})$ of $f$ to each $(\beta_i,\beta_{i+1})$ is of the form
$$
f\vert (\beta_i,\beta_{i+1}) = \left| C_i \log (x-\beta_i)\right| + \left| C_i \log (\beta_{i+1}-x)\right| + g_i(x)
$$
where $C_i \geq 0$ is a non-negative constant, $g_i$ is a function of bounded variation on $[\beta_i,\beta_{i+1}]$   and not all $C_i$ are simultaneously zero (see also the definitions in \S\ref{sec:roofs}). Thus, if $C_i\neq 0$, $f$ explodes logarithmically at each endpoint of $I_i$ and the \emph{singularities} are \emph{symmetric} (not necessarily $f\vert (\beta_i,\beta_{i+1})$ because of the presence of $g_i$). An example of a roof function is shown in Figure~\ref{symlog}.

 \begin{figure}[h!]
\includegraphics[width=0.6\textwidth]{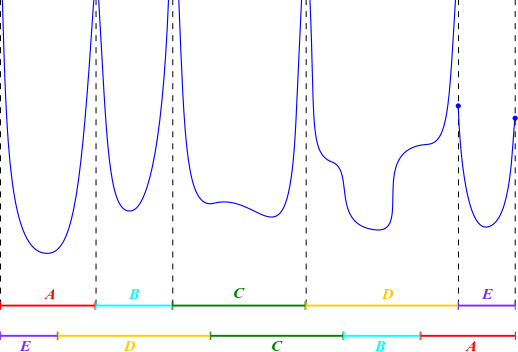}
 \caption{ A special flow over a symmetric $5$-IET with endpoints $\beta_0, \dots, \beta_5$ under a roof {$f \in \SymLog{\sqcup_{i=0}^4 I_i}$}. In this example $C_4=0$. \label{symlog}}
\end{figure}

The main result in the setting of special flows is the following.

\begin{thm}\label{thm:ss_sf}
Let $\pi$ be a symmetric permutation. For almost every IET $T$ with permutation $\pi$ and endpoints $\beta_i, 0\leq i \leq d$, for any $f\in \SymLog{\sqcup_{i=0}^{d-1} I_i}$, the special flow $(T_t^f)$ over $T$ under $f$ has purely singular spectrum.
\end{thm}

The result in the context of  special flows is hence more general (since it holds for IETs of any number $d\geq 2$ of exchanged intervals in the base), but unfortunately (similarly to the case of Scheglov's  result \cite{Sch:abs} on absence of mixing) this does not yield any general result for smooth locally Hamiltonian flows  on surfaces of genus higher than two  (see Remark~\ref{rk:realizable}). The role played by the symmetry of the IET, together with  the symmetry in the roof, is explained in \S\ref{sec:strategy}. We remark that the special case of Theorem~\ref{thm:ss_sf} for $d=2$ recovers the main result from \cite{Fr-Le0}.

\subsection{Translation surfaces well approximated by single cylinders}\label{sec:main_cyl}

We now state some results on cylinders in translation surfaces, which will be used as an ingredient in our proof of singularity of the spectrum but holds in more generality for any translation surface. As a reference on background material on translation surfaces, we refer the reader to one of the surveys \cite{FM,Vi,Yo}.

\smallskip
Let $(M, \omega)$ denote a (compact) translation surface, namely a Riemann surface $M$ with an Abelian differential $\omega$ which defines a flat metric with conical singularities on $M$, which correspond to zeros of $\omega$. Recall that the notion of direction is well defined globally on a translation surface, thus directions can be identified with $S^1$. Denote by $\mathcal{C}yl_\omega$ the set of all cylinders in the translation surface $(M,\omega)$, i.e.~$C\in \mathcal{C}yl_\omega$  is a maximal open annulus filled by
homotopic simple closed (flat) geodesics.
Any cylinder $C$ is isometric to an annulus $J\times \R/c\Z$, where $J\subset \R$ is an (open) interval and $c>0$. The \emph{core curve} of $C$ is the closed geodesic represented by $\{x\}\times \R/c\Z$, where $x$ is the mid point of the interval $J$.

For any cylinder $C\in \mathcal{C}yl_\omega$, denote by:
\begin{itemize}
\item $\gamma(C)$ the core curve of $C$;
\item $\theta_C\in S^1$ the \emph{direction} of $C$ (i.e. the direction of the core curve $\gamma(C)$);
\item $a(C)$ the area of $C$ with respect to the flat area-form induced by $\omega$;
\item $\ell(C)>0$ the length of $\gamma(C)$ in the flat metric.
\end{itemize}

Assume that $a(M)=1$. For every $0<\epsilon<1$ let ${\mathcal{C}yl}^\epsilon_\omega$  be the subset of cylinders $C\in{\mathcal{C}yl}_\omega$ with $a(C)\geq 1-\epsilon$. We are interested in showing that on a typical translation surface, a full measure set of directions can be approximated (with a certain speed) by the directions of a sequence of cylinders ${\mathcal{C}yl}^\epsilon_\omega$, i.e.\ by single cylinders of area close to one.

To state the result, let $\mathcal{C}$ denote a connected component  of (a stratum of) the moduli space  of compact area one translation surfaces.
In particular, all translation surfaces in $\mathcal{C}$ have the same number and type of conical singularities, or equivalently zeros of the Abelian differential. Let $m_\mathcal{C}$ denote  the sum of the multiplicities of singular points (for example $m_\mathcal{C}=2$ for translation surfaces with genus two and two simple saddles, more in general $m_\mathcal{C}=\sum_{i=1}^n{\kappa_i}$ for connected components of the stratum $\mathcal{H}(\kappa_1, \dots, \kappa_n)$).

Recall that each $\mathcal{C}$ is endowed by a natural volume probability measure $\nu_\mathcal{C}$ (the Masur-Veech measure \cite{Ma:int,Ve:gau}).
Let $c_1(\mathcal{C})$ be the corresponding Siegel-Veech constant (we refer e.g.~to \cite{Es-Ma} for the notion of Siegel-Veech constant, which enters in counting problems on translations surfaces).

Let $\lambda$ denote the Lebesgue (probability) measure on the unit circle $S^1$ in the complex plane, which we freely identify with $[0,2\pi)$. The main result of this section is the following.
\begin{prop}[Directions well approximated by large cylinders]\label{prop:cyl}
For $\nu_\mathcal{C}$-almost every translation surface $(M,\omega)\in\mathcal{C}$ and any $\epsilon>0$ there exists  a sequence of cylinders $(C_i)_{i\geq 1}$ on $(M,\omega)$ so that  $\ell({C_i})\to+\infty$ as $i\to+\infty$ and
for every $i\geq 1$ we have
\[a(C_i)\geq 1-\epsilon\quad\text{and}\quad\|\theta_{C_i}-\tfrac \pi 2\|<\frac 1 {\ell({C_i})^2 \log(\ell({C_i}))}.\]
\end{prop}

The sequence of cylinders $(C_i)_{i\geq 1}$  gives what we will call a \emph{good approximation} of the vertical direction by directions of single cylinders (when $\epsilon $ is small).  The approximation rate $\ell({C_i})^{-2} \log(\ell({C_i}))^{-1}$ is chosen to allow us to prove singularity of the spectrum in the genus two case.

\smallskip

Proposition~\ref{prop:cyl} can be easily deduced (see Section~\ref{sec:cylinder})  from the following  result on translation surfaces, which mimics, in the context of translation surfaces, the statement of Khintchine Theorem in Diophantine approximation.

\begin{thm}[Khintchine Theorem for cylinders on translation surfaces; c.f.\
{\cite[Theorem 1]{Ch}} and {\cite[Theorem 6.1 (2)]{Ma-Tr-We}}]\label{thm:Kcyl}  Let $\psi:\mathbb{R}^+ \to \mathbb{R}^+$ be non-increasing so that $t\psi(t)\leq 1$ for $t$ large enough and $\int_1^{+\infty} t \psi(t)=\infty$. Then for a.e.\ $(M,\omega)\in \mathcal{C}$ and every $0<\epsilon<1/2$ the set
\begin{equation}\label{eq:set}
W^\psi_\omega=\bigcap_{m\geq 1}\bigcup_{\{C \in {\mathcal{C}yl}^\epsilon_\omega  :\ \ \ell(C)\geq m\}} \big\{\phi\in S^1:  \|\theta_C-\phi\|<\psi(\ell({C}))\big\}
\end{equation}
has full Lebesgue measure. Moreover, for a.e.\ $(M,\omega)\in \mathcal{C}$ there exists a sequence $(C_i)_{i\geq 1}$ in ${\mathcal{C}yl}_\omega^\epsilon$
such that $\ell({C_i})\to+\infty$ as $i\to+\infty$ and $\|\theta_{C_i}-\frac{\pi}{2}\|<\psi(\ell({C_i}))$ for all $i\geq 1$.
\end{thm}
This result's proof is independent of the rest of the paper and follows from the methods of \cite{Ch} and \cite{Ma-Tr-We}. It is proved in Section~\ref{sec:cylinder}.


\subsection{Strategy of the proof of the main result}\label{sec:strategy}
Let us conclude the introduction explaining the main ideas in the proof. To study ergodic and spectral properties of locally Hamiltonian flows, it is standard to exploit their representation as special flows over an IET (or a rotation when $g=1$). The growth of Birkhoff sums $S_n(f)= \sum_{k=0}^{n-1} f\circ T^k$ of the roof function $f$ and its derivatives play a crucial role in the proof of properties such as mixing, weak mixing, multiple mixing, shearing properties and disjointness phenomena among others. Spectral behavior is no exception, but requires  a \emph{much} more delicate understanding of \emph{weak limits} of Birkhoff sums.

The \emph{criterion} we use for proving \emph{singularity of the spectrum} of special flows (stated in \S\ref{sec:sc}) is devised to deal with flows which display \emph{absence of mixing}.  An important early criterion for absence of mixing appears in Katok's work \cite{Ka:int}, which shows 
 that special flows over IETs under roof functions of bounded variation are never mixing, and by Kochergin's, which shows the absence of mixing for special flows over rotations under a roof with a symmetric logarithmic singularity (see \cite{Ko:abs, Ko:07}). Both criteria require as input  \emph{tightness} of Birkhoff sums along some subsequences of \emph{rigidity} (or \emph{partial rigidity}) \emph{times}, i.e.\ one has to show that there exists a sequence $(q_n)$ of times such that $T^{q_n}$ converges to identity on subsets $E_n $ of measure tending to one (if there is rigidity, or measure bounded below in the case of partial rigidity) and at the same time, for some centralizing sequence $(a_n)$ and uniform constant $C$,
 $Leb\{ x\in E_n | \ |S_{q_n}(f)(x) - a_n|<C\}/Leb(E_n) \to 1 $.
In the case of rotations and functions of bounded variation, this follows easily from Denjoy-Kosma inequality, while for functions with symmetric logarithmic singularities one has to exploit a \emph{cancellation} phenomenon among contributions coming from the symmetric singularities.

These type of criteria were pushed in two different directions in \cite{Fr-Le0} and \cite{Sch:abs, Ul:abs}. Fr\k{a}czek and Lema\'nczyk in \cite{Fr-Le0}, considering the same example as Kochergin (special flows with one symmetric logarithmic singularity over rotations), showed that if, in addition to \emph{tightness}, one can also control the \emph{tails} of the distribution of the centralized Birkhoff sums $S_{q_n}(f)(x) - a_n$, one can prove much stronger results (using joinings and Markov operators) and deduce in particular spectral disjointness from mixing flows, which implies that the spectrum is purely singular.  In \cite{Sch:abs, Ul:abs} IETs were considered on the base (which is required when treating surfaces of genus $g\geq 2$). In this case, cancellations are much more difficult to prove because of the absence of the Denjoy-Koksma inequality. To prove absence of mixing, though, it is sufficient to prove cancellations on carefully constructed partial rigidity times. The usual tool to study IETs (which is \emph{not} used in this paper) is Rauzy-Veech induction, a renormalization algorithm for IETs. In \cite{Ul:abs} Rauzy-Veech induction (and the log integrability of the associated cocycle) are heavily used to obtain cancellations at carefully chosen renormalization times.
 On the other hand, in  Scheglov's work \cite{Sch:abs}, the cancellations were proved through a careful combinatorial analysis of the substitutions arising from the action of Rauzy-Veech induction on symmetric permutations.
Ideally one would like to  combine these two approaches in order to prove spectral results (as  in \cite{Fr-Le0})  for IETs (as in \cite{Sch:abs, Ul:abs}). The key difficulty is that cancellations are hard to achieve for IETs on sets of large measure (the cancellations in \cite{Ul:abs} for example are crucially based on \emph{balanced} Rauzy-Veech induction time, which are opposite to rigidity times).

In this paper, for surfaces of genus two or symmetric permutations, we (implicitly) exploit  a  very geometric approach to deduce cancellations, based on  a simple  mechanism which uses in an essential way the hyperelliptic involution: the key idea is that, for any symmetric (of equal backward and forward length) trajectory from a fixed point of the hyperelliptic involution, there are perfect cancellations for Birkhoff sums of the derivative of the roof function (see \S\ref{sec:symmetry}). Cancellations achieved through the hyperelliptic involution have the advantage of being compatible with rigidity. In particular, they can be shown to hold for Birkhoff sums along a rigidity tower of area close to one (i.e.\ a Rokhlin tower for the IET which comes from a cylinder of area close to one on the surface).

One of the advantages of this approach is that we do not make use at all of Rauzy-Veech induction. Theorem~\ref{thm:singsp} also provides an independent proof of Scheglov's work \cite{Sch:abs}, which highlights the role played by the hyperelliptic symmetry in  Scheglov's combinatorial calculations.

In order to prove singularity of the spectrum using this approach (and the criterion stated in Section~\ref{sec:singcrit}, which is a generalization of the criterion in \cite[Corollary 5.2]{Fr-Le1} and \cite[Proposition 11]{Fr-Le0}) though, another ingredient is needed, namely \emph{good rigidity} (see Definition~\ref{def:IETrecurrence}). Cancellations achieved thanks to the hyperelliptic involution only hold for Birkhoff sums along a full rigidity tower. To prove the exponential tails estimates needed to apply the criterion on the whole tower, one has to controldeincomplete sums, that can in general fail to be tight. These  potentially worse  estimates (see Remark \ref{rk:exptails}) are compensated for by assuming that points in the base of the rigidity tower have a quantitatively good form of recurrence (see Definition~\ref{def:IETrecurrence}).  The existence of good rigidity towers for almost every IET is deduced (in \S\ref{sec:final}) from the abundance of translation surfaces well approximated by single cylinders (i.e.~from Proposition~\ref{prop:cyl}).


\subsection{Structure of the paper}
In Section~\ref{sec:locHam_sf} we first recall some background material on locally Hamiltonian flows and their reduction to special flows, with particular attention to the form of the representation in the special case of genus two and two isomorphic saddles (see Lemma~\ref{lemma:reduction} and Corollary~\ref{cor:reduction}).  
Our criterion for singularity for special flows (Theorem~\ref{thm:singcrit}) is stated and proved in \S\ref{sec:sc}, after recalling basic spectral notions in \S\ref{sec:spectral}. Elementary but precise estimates on (Birkhoff sums of) functions with symmetric logarithmic singularities
are proved in  \S\ref{sec:BS}; these, combined with the symmetry and the cancellation arguments are explained in \S\ref{sec:symmetry} (which  follow from the hyperelliptic involution, see Lemmas \ref{lem:symmetries} and \ref{lem:cancellations}), are then used in \S\ref{sec:final}, in combination with the rigidity deduced from single cylinders (given by Proposition~\ref{prop:cyl}) to conclude the proof of the singularity result in genus two (i.e.~Theorem~\ref{thm:singsp}).
Finally, in Section~\ref{sec:cylinder} (which can be read independently), we prove the Khintchine-type result for translation surfaces (Theorem~\ref{thm:Kcyl} and show how it implies Proposition~\ref{prop:cyl} about translation surfaces well approximated by single cylinders.

\section{Locally Hamiltonian flows and reduction to special flows}\label{sec:locHam_sf}
In this section we recall some definitions, basic notions and background material on locally Hamiltonian flows (\S\ref{sec:locHam} and \S\ref{sec:measureclass}) and on  special flows \S\ref{sec:sfdef}. We also quickly summarize some results in the literature of locally Hamiltonian flows \S\ref{sec:history}.

\subsection{Smooth area-preserving flows as locally Hamiltonian flows}\label{sec:locHam}
In this section we define locally Hamiltonian flows and show that they are equivalent to smooth area-preserving flows.

\smallskip
Assume that $M$ is a $2$-dimensional closed connected orientable smooth surface of genus $g\geq 1$.
  Let $X:M\to TM$ be a smooth tangent vector field with finitely many fixed points
and such that the corresponding flow $(\varphi_t)_{t\in\R}$ preserves a smooth volume form $\omega$ (which is locally given by $V(x,y) d x \wedge d y$ for some smooth positive real valued function  $V:U \to \mathbb{R}$ on the coordinate chart). Then, letting $\eta:=\imath_X\omega=\omega( \eta, \,\cdot \,)$, where $\imath_X$ denotes the contraction operator, we have $d\eta=0$. Furthermore,  since $\eta$ is a smooth closed 1-form, for any $p\in M$ and any simply connected neighbourhood $U$ of $p$ there exists a smooth (local Hamiltonian) map (unique up to additive constant)
such that $dH=\eta$ on $U$.

\smallskip

Conversely, let  $(M, \omega)$ be a 2-dimensional symplectic manifold, where $M$ is a  closed connected orientable smooth surface of genus $g\geq 1$
endowed with the standard area form $\omega$ (obtained as pull-back of the area form ${d} x \wedge {d} y$ on $\mathbb{R}^2$). 
Let $\eta $ be a smooth closed real-valued differential $1$-form.   Let $X$ be the vector field determined by $\eta = \imath_X \omega $ and consider the flow $(\varphi_t)_{t\in\mathbb{R}}$ on $M$ associated to $X$. Since $\eta$ is closed, the transformations $\varphi_t$, $t \in \mathbb{R}$, are  area-preserving (i.e.\ preserve the area form $\omega$ and the measure given by integrating it). We will always assume that the form is normalized so that the associated measure gives area $1$ to $M$.

The flow $(\varphi_t)_{t\in\mathbb{R}}$  is known as the \emph{multi-valued Hamiltonian} flow associated to $\eta$. Indeed, the flow $(\varphi_t)_{t\in\mathbb{R}}$ is \emph{locally Hamiltonian}, i.e.\ \emph{locally} one can find coordinates $(x,y)$ on $M$ in which $(\varphi_t)_{t\in\mathbb{R}}$ is given by
 the solution to the  equations $\dot{x}={\partial H}/{\partial y}$, $\dot{y}=-{\partial H}/{\partial x}$ for some smooth  real-valued Hamiltonian function $H$.  A \emph{global}  Hamiltonian $ H$ cannot be in general defined (see \cite{NZ:flo}, \S1.3.4), but one can think of  $(\varphi_t)_{t\in\mathbb{R}}$ as globally given by a \emph{multi-valued} Hamiltonian function.

When $g\geq 2$, the  (finite) set of fixed points of $(\varphi_t)_{t\in\mathbb{R}}$ is always non-empty. We will always assume that $1$-form $\eta$  is \emph{Morse}, i.e.\ it is locally the differential of a Morse function.  Thus, zeros of $\eta$ are isolated and finite and all  correspond to either centers (see Figure~\ref{island}) or simple saddles (see Figure~\ref{simplesaddle}), see \S\ref{sec:singularities} (as opposed to degenerate \emph{multi-saddles} which have $2k$ separatrices for $k>2$, see Figure \ref{multisaddle}).


\subsection{Topology and measure class on locally Hamiltonian flows}\label{sec:measureclass}
One can define a \emph{topology} on locally Hamiltonian flows by considering perturbations of closed smooth $1$-forms by smooth closed $1$-forms. {With respect to this topology, the set of locally Hamiltonian flows   whose zeros are all Morse (hence isolated and finite, simple saddles or centers) is open and dense (and hence in particular \emph{generic} in the Baire category sense), see for example Lemma 2.3 in \cite{Ra:mix}}.  Let $\Sigma$ be the set of fixed points of $\eta$  and let $k$ be the cardinality of $\Sigma$.

\smallskip
{The \emph{measure-theoretical notion of  typical} that we use is defined as follows and coincide with the notion of typical induced by the  \emph{Katok fundamental class} (introduced by Katok in \cite{Ka:inv}}, see also \cite{NZ:flo}).  {We recall that two measures belong to the same measure class if they have the same sets of zero mesure (and hence induce the same notion of full measure, or \emph{typical})); thus, a \emph{measure class} is uniquely identified by a collection of sets which have measure zero with respect to all measures in the class.}
  Let $\gamma_1, \dots, \gamma_n$ be a base of the relative homology $H_1(M, \Sigma, \mathbb{R})$, where $n=2g+k-1$ ($k:=\#\Sigma$). The image of  $\eta$ by the period map $Per $ is {$Per(\eta) = (\int_{\gamma_1} \eta, \dots, \int_{\gamma_n} \eta) \in \mathbb{R}^{n}$. The pull-back $Per_* Leb$ of the Lebesgue measure class by the period map gives a measure class on closed $1$-forms (with $k$ critical points): explicitely, the measure zero sets for this measure class are all preimages through $Per$ of measure zero sets in  $\mathbb{R}^{n}$ (with respect to the Lebesgue measure $Leb$ on $\mathbb{R}^{n}$).
We say that a property is \emph{typical}  if it is satisfied for a set of locally Hamiltonian flows a \emph{full measure}, namely the complement of a measure zero set for this measure class.} 

\smallskip
A \emph{saddle connection} is a flow trajectory from a saddle to a saddle and a \emph{saddle loop} is a saddle connection from a saddle to the same saddle (see Figure \ref{island}). Notice that if the set of fixed points $ \Sigma $ contains a center, the island of closed orbits around it is automatically surrounded by a saddle loop homologous to zero (see Figure~\ref{island}). The set of locally Hamiltonian flows which have at least one saddle loop \emph{homologous to zero} form an \emph{open} set\footnote{Saddle loops homologous to zero are indeed persistent under small perturbations, see \S2.1 in \cite{Zo:how} or Lemma 2.4 in \cite{Ra:mix}.}. Flows in this open set decompose into several \emph{minimal components}\footnote{\emph{Minimal components}  are subsurfaces (possibly with boundary) on which the (restriction of the) flow is \emph{minimal}, in the sense that all semi-infinite trajectories are dense. As proved independently  by Maier \cite{Ma:tra}, Levitt \cite{Le:feu} and Zorich \cite{Zo:how}), each smooth area-preserving flow can be decomposed into up to  $g$  {minimal components} and  {periodic components}, i.e.~subsurfaces (possibly with boundary) on which all orbits are closed and periodic (as the disk filled by periodic orbits in Figure \ref{island}).}.
On the other hand, in the open set
$\mathscr{U}_{min}$ consisting of  locally Hamiltonian flows with only simple saddles and no saddle loops homologous to zero
 a typical flow (in the measure theoretical sense defined above) has no saddle connections and hence it is \emph{minimal}  by a result of  Maier \cite{Ma:tra} (or, in the language of special flows introduced in the next section, by the result of Keane~\cite{Keane} on IETs).

\begin{rem}\label{rem:reparametrization} Minimal locally Hamiltonian flows (as well as minimal components) can be seen as (singular) time-reparametrizations of \emph{translation flows} (linear flows on translation surfaces), i.e.\ they have the same orbits as a translation flow, but the movement along the orbits happens with different speed (and in particular it takes an infinite time to reach saddles). This follows for example from a result in \cite{Ma:tra}, which guarantees that any $1$-form $\eta$ without saddle loops homologous to zero is the real part of a holomorphic one form (see \cite{Zo:how}).
\end{rem}

\subsection{Singularities and normal forms}\label{sec:singularities}
 In this section we associate  an invariant to each non-degenerate fixed point which will play a crucial role in describing isomorphic singularities and their special flow representation.

\smallskip
Let $M$ be an $m$-dimensional $C^2$-manifold equipped with a volume form $\omega$. Let $f:M\to\R$ be a $C^2$-map whose critical points are isolated. Suppose that $p\in M$
is a critical point of $f$ and let us consider local coordinates $(x_1,\ldots, x_m)$ in a neighbourhood of $p$ so that
$(0,\ldots,0)$ are local coordinates of $p$. In these local coordinates $\omega_{(x_1,\ldots, x_m)}=V(x_1,\ldots, x_m)\,dx_1\wedge\ldots\wedge d x_m$,
where $V$ is a positive (or negative) function. Let
\[K_\omega(f,p):=\frac{\det\operatorname{Hess}(f)(0,\ldots,0)}{V^2(0,\ldots,0)}.\]
Since $df(p)=0$, $K_\omega(f,p)$ does not depend on the choice of local coordinates
and it imitates the notion of curvature (on the graph of $f$) at any critical point of $f$ even if $M$ is not equipped with any Riemannian metric.
Moreover, $K_\omega(f,p)\neq 0$ if and only if $p$ is a non-degenerate critical point and by the Morse lemma there exist
local coordinates $(x_1,\ldots, x_m)$ in a neighbourhood of $p$ such that
\[f(x_1,\ldots, x_m)=f(0,\ldots,0)-x^2_1-\ldots-x^2_k+x^2_{k+1}+\ldots+x^2_m\]
and  $\sgn K_\omega(f,p)=(-1)^k$.

\smallskip
Assume now that $M$ is two dimensional and  consider  a local Hamiltonian $H: U \to \mathbb{R}$, $U \subset M$, of a locally Hamiltonian flow $(\varphi_t)_{t\in\R}$ preserving the area form $\omega$.
 If $p$ is a fixed point of $(\varphi_t)_{t\in\R}$ (hence a critical point of $H$), then we can define
\[K_{\omega,X}(p):=K_\omega(H,p).\]
The quantity $K_{\omega,X}(p)$ does not depend on the choice of local Hamiltonian, hence it is well defined.

A fixed point $p$ is \emph{non-degenerate} exactly when $K_{\omega,X}(p)\neq 0$.
If $K_{\omega,X}(p)>0$ then $p$ is the centre of a topological disc filled with periodic orbits, as in Fig.~\ref{island}. If $K_{\omega,X}(p)<0$ then $p$ is a saddle point (see Fig.~\ref{simplesaddle}).

 \begin{figure}[h!]
  \subfigure[\label{island} Center]{
  \includegraphics[width=0.23\textwidth]{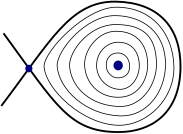} 	} \hspace{3mm} \subfigure[Saddle\label{simplesaddle}]{ \includegraphics[width=0.18\textwidth]{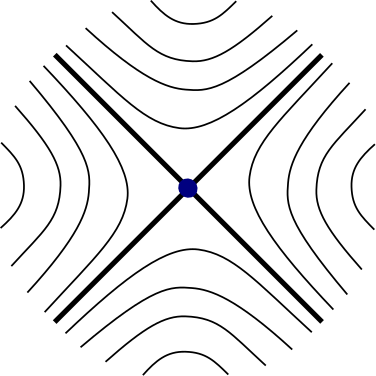}}\hspace{3mm}
\subfigure[Multisaddle\label{multisaddle}]{
\includegraphics[width=0.18\textwidth]{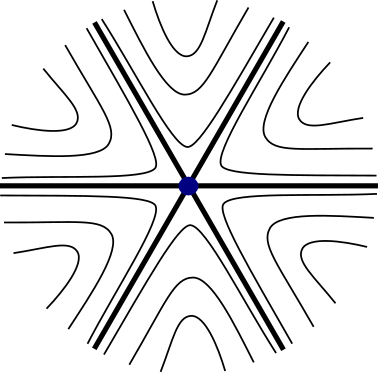} }	\hspace{3mm}
 \caption{Type of non-degenerate fixed points for an area-preserving flow.\label{saddles1}}
\end{figure}

\subsection{{Isomorphic saddles.}}\label{sec:isomorphic}
We will use the following working definition of \emph{isomorphic} (simple) \emph{saddles}.
\begin{defn}[Isomorphic saddles]\label{def:isomorphic} We say
that two saddles corresponding to fixed points $p_1,p_2$ of  $(\varphi_t)_{t\in\R}$ are \emph{isomorphic} iff $K_\omega(H,p_1)= K_\omega(H,p_2)$.
\end{defn}
Indeed, the above definition implies that, for both $i=1,2$, we can find local coordinates around $p_i$ so that $p_i$ is mapped to $(0,0)$, $\omega$ is given by the standard form $ dx \wedge dy$, and the local Hamiltonian has the form
$H (x,y) = K x y + \text{higher order terms},$
for a common value $K:= \sqrt{-K_\omega(H,p_1)}=\sqrt{-K_\omega(H,p_2)}$.  This property is satisfied if  there is a smooth symplectic (preserving $\omega$) isomorphism mapping flow trajectories to flow trajectories among the two local neighbours.{
\begin{defn}[Isomorphic saddles locus] \label{def:isomorphiclocus}
 We will denote by $\mathcal{K}$ the set of locally Hamiltonian flows on a surface of genus two in $\mathcal{U}_{min}$ which have two isomorphic simple saddles.
\end{defn}
The notion of \emph{typical} on $\mathcal{K}\subset \mathcal{U}_{min}$ (which is the notion used in the statement of Theorem~\ref{thm:singsp}) is obtained restricting the notion of Katok measure class (see \S\ref{sec:measureclass}) to $\mathcal{K}$ as follows. Consider the period map $Per: \mathcal{K}\to \mathbb{R}^n$ obtained restricting the period map $Per: \mathcal{U}_{min}\to \mathbb{R}^n$ defined in  \S\ref{sec:measureclass} to $\mathcal{K}\subset \mathcal{U}_{min}$.   We say that a property holds for a \emph{typical} flow in $\mathcal{K}$ if it \emph{fails} on a set of measure zero  with respect to the pull back of the Lebesgue measure class via $Per: \mathcal{K}\to \mathbb{R}^n$, namely it fails on the preimage $Per^{-1}(Z)$ of a set $Z\subset \mathbb{R}^n$ with $Leb(Z)=0$. See also Remark~\ref{rem:fullmeasure} for a reformulation of this notion of typical in terms of special flows representations.
}

\subsection{Special flows}\label{sec:sfdef}
Let us now recall the definition of \emph{special flow}.  Let $T$ be an automorphism of a standard (Borel)  probability space $(X,\mathcal{B},\mu)$. Let $f:X\to\R_{>0}$ be an integrable function  so that $\inf_{x\in X} f(x)>0$. Let us denote by $S_n(f)(x)$ the  \emph{Birkhoff sum} defined by
\[S_n(f)(x)=\left\{
\begin{array}{rcl}
\sum_{0\leq i<n}f(T^ix)&\text{if}& n\geq 0\\
-\sum_{n\leq i<0}f(T^ix)&\text{if}& n< 0.
\end{array}
\right.\]
The \emph{special flow} $(T^f_t)_{t\in\R}$ built \emph{over} the automorphism $T$ and \emph{under} the \emph{roof function} $f$  acts
on
\[X^f:=\{(x,r)\in X\times\R: 0\leq r<f(x)\}\]
so that
\[T^f_t(x,r)=(T^nx,r+t-S_n(f)(x)),\]
where $n=n(t,x) \in\Z$ is a unique integer number with $S_n(f)(x)\leq r+t<S_{n+1}(f)(x)$.  Under the action of $(T^f_t)_{t>0}$,  a point $(x,y) \in X_f$ moves  with unit velocity along the vertical line up to the point $(x,f(x))$, then jumps instantly to the point $\left( T(x),0 \right)$, according to the base transformation and afterward it continues its motion along the vertical line until the next jump and so on.  The integer $n(t,x)$ (for $t>0$) is the number of discrete iterations of the map $T$ undergone by the orbit of $x$ up to time $t$.

The flow $(T^f_t)_{t\in\R}$ preserves the finite measure $\mu^f$ which is the restriction of $\mu\times\Leb_\R$ to $X^f$.
If $T$ is ergodic with respect to $\mu$, it is easy to see then  $(T^f_t)_{t\in\R}$ is also ergodic (with respect to  $\mu^f$), see e.g.~\cite{CFS:erg}.

\subsection{Roofs with logarithmic singularities}\label{sec:roofs}
 We now define the class of functions which we work with and arise as roof functions of locally Hamiltonian flows with non-degenerate saddles.

Let $T$ be an IET with endpoints of the continuity intervals $0:=\beta_0<\beta_1< \dots \beta_{d-1}<\beta_d:=|I|$ (see ~\S\ref{sec:main_sf}).
\begin{defn}[logarithmic singularities]\label{def:SymLog}
We say that a function $f $ has \emph{pure logarithmic singularities} at the endpoints $\beta_i$ of $T$ and write  $f \in \pLog{\sqcup_{i=0}^{d-1} I_i}$ if it is of the form
\begin{equation}\label{eq-1}
f(x)=\sum_{0\leq i<d}\chi_{(\beta_i,\beta_{i+1})}(x) \big(-C_i^+\log(x-\beta_i)-C_{i+1}^-\log(\beta_{i+1}-x)\big),
\end{equation}
for some constants $C_i^{\pm}\geq 0$, not all simultaneously zero. Notice that the signs are chosen so that $f\geq 0$.

We say that $f$ has pure \emph{symmetric} logarithmic singularities at the endpoints $\beta_i$ of $T$ and write $f \in {\pSymLog{\sqcup_{i=0}^{d-1} I_i}}$  if in addition  we have that $C_i^+=C_{i+1}^-$ for $0\leq i< d$, so that the function is \emph{symmetric} on each interval $(\beta_i,\beta_{i+1})$.

We say that $f$ has \emph{logarithmic singularities} (resp. \emph{symmetric logarithmic singularities}) and write $f \in \Log{\sqcup_{i=0}^{d-1} I_i}$ (resp.\  $f \in  \SymLog{\sqcup_{i=0}^{d-1} I_i}$) if and only if $f$ can be written as $f=f^p+g$ where $f^p \in \pLog{\sqcup_{i=0}^{d-1} I_i}$ (resp.\  $f \in \pSymLog{\sqcup_{i=0}^{d-1} I_i}$) has pure logarithmic (symmetric) singularities and $g:I\to\R$ is a function of bounded variation.
\end{defn}
 We remark that we allow some of the $C_i^\pm$ to be zero; so $f$ could have a finite one-sided limit at some $\beta_i$ (but we assume that at least one of the singularities is indeed logarithmic).
We notice also that this symmetry condition (which is symmetric on \emph{each} exchanged interval, i.e.\ a function   $f \in \SymLog{\sqcup_{i=0}^{d-1} I_i}$) is symmetric on each continuity interval $(\beta_i,\beta_{i+1})$ is not the same than appears in other  works on locally Hamiltonian flows with non-degenerate saddles (where {symmetric logarithmic singularities} refers to functions in  ${\Log{\sqcup_{i=0}^{d-1} I_i}}$ such that $\sum_{i=0}^{d-1}C_i^+= \sum_{i=1}^{d}C_i^-$).  We will use the assumption that the saddles are isomorphic in Theorem~\ref{thm:singsp} to obtain this stronger form of symmetry for such (genus 2) surfaces. 

\subsection{Reduction to symmetric special flows}\label{sec:reduction}
It is well known that minimal (or minimal components of) locally Hamiltonian flows   can be represented as special flows over  rotations (in genus one) or interval exchange transformations (for $g\geq 2$). The roof function has a finite number of singularities (where it explodes to infinity) which are of \emph{logarithmic-type} (see the form of singularities in Def.~\ref{def:SymLog})  if the fixed points are simple saddles or \emph{power-type} singularities  (i.e.~singularities of the form $C_i^{\pm}/\vert x-\beta_i \vert^{\alpha_i}$ for some power $0<\alpha_i<1$)  in presence of (degenerate) multi-saddles (as in Figure~\ref{multisaddle}) or stopping points.  In case of minimal flows with only simple saddles (or more in general when there are no saddle loops homologous to zero), the logarithmic singularities display a form of \emph{symmetry}\footnote{Symmetry, or asymmetry, of the logarithmic singularities are crucial in determining  the mixing properties of the flow, see \S\ref{sec:history}. Asymmetry is usually introduced by  the presence of saddle loops homologous to zero, we refer for example \cite{Ra:mix} for details.}.
 For Theorems~\ref{thm:singsp} and \ref{thm:ss_sf} we require  a stronger form of symmetry for both the roof and the base transformation.

\smallskip
The following Lemma provides the reduction to symmetric special flows which we need to prove the result on flows in genus two (see in particular Corollary~\ref{cor:reduction}). While ($ii$) is standard (and included only for completeness),  ($i$) and ($iii$) provide the required more detailed information on the symmetry of the base and the roof (in particular the precise values of the constants $C_i^\pm$).

\begin{lem}[Symmetries of the reduction to special flows]\label{lemma:reduction}
Let $(\varphi_t)_{t \in \mathbb{R}}$ be a \emph{minimal} locally Hamiltonian flow on a surface $M$.  Then,  $(\varphi_t)_{t \in \mathbb{R}}$ is \emph{measurably isomorphic} to a special flow $(T^f)_{t \in \mathbb{R}}$ over an IET $T$ (whose endpoints are denoted by $\beta_i$, $0 \leq i   \leq d$)  under a roof function $f:I  \to \mathbb{R}_{>0}\cup \{ + \infty \}$. The special flow representation can be chosen to that:
\begin{itemize}
\item[$(i)$] if $M$ has genus $1$ or $2$ then 
 $T$ is a $d$-IET given by a  \emph{symmetric permutation} (with $d= 2g$ if there is a unique saddle or $d=2g+1$  if there are two);
\item[$(ii)$] when  $(\varphi_t)_{t \in \mathbb{R}}$ has only  \emph{simple} saddles, $f\in \Log{\sqcup_{i=0}^{d-1}I_i}$;
\item[$(iii)$] under the assumptions of ($ii$), the constants $C_i^\pm$ in \eqref{eq-1} are given by the values of the invariants $K_{\omega,X} (p)$ associated to saddle points: if the forward $(\varphi_t)_{t \in \mathbb{R}}$-orbit of $\beta_i$ meets the saddle point $p$ before returning to $I$, then
\begin{equation} \label{saddle_constants} C_i^+=C_{i}^-=\frac{1}{\sqrt{-K_{\omega,X}(p)}}.\end{equation}
\end{itemize}
\end{lem}

The proof of Lemma~\ref{lemma:reduction} is presented below. Combining $(i)-(iii)$ of  Lemma~\ref{lemma:reduction}, we have the following Corollary which we will use to prove Theorem~\ref{thm:singsp}.
\begin{cor}\label{cor:reduction}
If  $M$ has genus two and  $(\varphi_t)_{t \in \mathbb{R}}$ is a minimal locally Hamiltonian flow with  two isomorphic simple saddles $p_1,p_2\in M$, then it is isomorphic to a special flow over an IET $T$ with a symmetric permutation $\pi$ with $d=5$ and roof  $f\in \SymLog{\sqcup_{i=0}^{4}I_i}$. More precisely, there exists $0\leq i_0<5$ such that $C^+_{i_0}=C^-_{i_0+1}=0$ and
\[C_i^+ =C_{i+1}^-=\frac{1}{\sqrt{-K_{\omega,X}(p_1)}}=\frac{1}{\sqrt{-K_{\omega,X}(p_2)}}>0\quad\text{ for all }\quad i\neq i_0.\]
\end{cor}

\begin{rem}\label{rk:realizable}
The conclusion of Part ($i$) of Lemma~\ref{lemma:reduction} also holds more in general when the flow $(\varphi_t)_{t \in \mathbb{R}}$ has one or two saddles and is the time-change of a linear flow on a translation surface  $M$ in a hyperelliptic component of stratum of the form $\mathcal{H}(2g-2)$ or $\mathcal{H}(g-1,g-1)$, $g\geq 1$ (the  saddles then  have respectively $4g-2$ or $2g$ and $2g$ separatrices). On the other hand,  to have a roof $f\in \mathcal{S}ym\mathcal{L}og$ one needs by Part ($ii$) to have only simple saddles (with $4$ separatrices), so this forces $g=2$ and two singularities.  Thus, the special flows in Theorem~\ref{thm:ss_sf} arise as special representation of \emph{minimal} smooth surface flows only in genus two.
Finally,  notice  also that the assumption that the two saddles are isomorphic is needed to have the symmetry of the constants.
\end{rem}

\begin{proof}[Proof of Lemma~\ref{lemma:reduction}]
Since  $(\varphi_t)_{t\in\mathbb{R}}$ is minimal, any curve $\gamma$ transverse to $(\varphi_t)_{t\in\mathbb{R}}$ is a \emph{global} transversal (i.e.\ intersects all infinite orbits) and hence provides a Poincar{\'e}  section for $(\varphi_t)_{t\in\mathbb{R}}$. Let us say that the parametrization of $\gamma $ is   \emph{standard} if $\gamma:I\to M$ (where $I$ is an interval starting at zero) is parametrized  so that $\eta(d\gamma)=1$. It is well known (see for example \cite[Section 4.4]{Yo}) that, in
the standard parametrization, the Poincar{\'e}  first return map $T:I\to I$ to $\gamma$ is an IET. The number of exchanged intervals is  $d=2g+k-1$ (where $k$ is the cardinality of the set of fixed points) if the endpoints of $\gamma$ are chosen on separatrices and, if $0=\beta_0<\beta_1<\ldots<\beta_d=|I|$
denote  the endpoints of exchanged intervals, the (forward) trajectories from all the $\beta_i$'s are separatrices which end in  a saddle (notice  that there are no centers since $(\varphi_t)_{t\in\mathbb{R}}$  is minimal) and do not return to $I$, with the exception of two of them,  which first return to the endpoints of $\gamma$  or its backward trajectory is a separatrix which starts from a saddle.

To prove ($i$), it is convenient to recall that  $(\varphi_t)_{t\in\mathbb{R}}$ is a time-change of a translation flow $(h_t)_{t\in\mathbb{R}}$ (see Remark~\ref{rem:reparametrization}). Let us denote by $(M, \omega)$ the translation structure on $M$. If $M$ has genus one or two, then it belongs to one of the strata $\mathcal{H}(0),$ $\mathcal{H}(2)$ or $\mathcal{H}(1,1)$ 
  and it admits a \emph{hyperelliptic involution}, i.e.\ there is a diffeomorphism $\iota: M \to M$ (affine in the coordinate charts of $(M, \omega)$) such that $\iota^2 =Id$. Let us choose $\gamma$ so that $\gamma(I)$ is an interval in $(M,\omega)$ and the image $\gamma(x_0)$ of the midpoint $x_0=\vert I \vert /2$ of $I$ is a \emph{Weierstrass point}, i.e.\ a fixed point of $\iota$, i.e.~$\iota(\gamma(x_0))=\gamma(x_0)$.  Thus $\iota$ fixes $\gamma$:
 let us denote $S: I \to I$ the symmetry such that $\iota (\gamma(x)) = \gamma(S(x))$.  Moreover, $\iota$ inverts the direction of trajectories of $(h_t)_{t\in\mathbb{R}}$, so that we have $h_{t}(\iota (q)) =\iota (h_{-t}(q))$ for all $q \in \gamma, t>0$.  Observe that this implies that the \emph{backward} trajectory from $q$  first returns to $\gamma$ in $p$ iff  the \emph{forward} trajectory from $\iota (q)$ first return to $\gamma$ in $\iota(p)$.

 Let $q\in\gamma$ be the first  return of the forward trajectory of $p\in \gamma$ to $\gamma$; if $x,y\in I$ are such that $\gamma(x)=p, \gamma(y)=q$,  since $T$ is by definition the first return map in the coordinates on $I$, this means that $T(x) = y$. Remark that equivalently $p$ is the first return of the \emph{backward} trajectory from $q$ to $\gamma$. Applying $\iota$, $\iota(p),\iota(q)$ have coordinates respectively $S(x), S(y)$ and, by the observation in the previous paragraph, the first return of the forward trajectory from $\iota (q)$ to $\gamma$ is $\iota(p)$ (since $p$ as just remarked is the first backward return of $q$). In coordinates, this can be written as $ T (S(y)) = S(x)$. Combining both equations in coordinates and recalling that $S^2=id$, we get
\begin{equation}\label{IETsymmetry}
S(x) = T(S(y)) = T(S (T(x)))  \qquad \Leftrightarrow \qquad T(x)= S\circ T^{-1}\circ S(x),
\end{equation}
for all $x \in I$.  Since in the translation structure $S:I \to I$ is an affine symmetry and it fixes $x_0=\vert I \vert /2$, $S$ must be of the form $S(x)= \vert I \vert -x$. One can then show that \eqref{IETsymmetry} forces the $T$ to be symmetric, i.e.\ the permutation $\pi$ must reverse the order of the intervals.
This concludes the proof of ($i$).

\smallskip
By standard ergodic theory (see e.g.~\cite{CFS:erg}), $(\varphi_t)_{t\in\mathbb{R}}$ is metrically isomorphic to the special flow over its Poincar{\'e} section $T$ under the function $f$ given by the first return time.
If all saddles are \emph{simple}, by the local form of Hamiltonian saddles (as first remarked by Arnold in \cite{Ar:top}, see also \cite[\S~7.1]{Co-Fr})  the first return time function $f:I\to\R_{>0}\cup \{+\infty\}$ is given by $f\in \Log{\sqcup_{i=0}^{d-1} I_i}$, i.e.~it has the form
\[f(x)=\sum_{0\leq i<d}\big(-C_i^+\log(x-\beta_i)-C_{i+1}^-\log(\beta_{i+1}-x)\big)\chi_{(\beta_i,\beta_{i+1})}(x)+g(x),\]
where $g:I\to\R$ is of bounded variation. This concludes the proof of $(ii)$.


Let us now show that if $\gamma(\beta_i)$ is the first backward hitting point of  a separatrix incoming to a saddle $p$ to $\gamma(I)$ then  $(iii)$  hold.
Choose local coordinates $(x,y)$ in a neighbourhood $U$ of $p$ and a local Hamiltonian so that $H(x,y)=xy$. Then $\omega(x,y)=V(x,y)dx\wedge dy$, where $V$ is a positive (or negative)
smooth map. Fix $\vep>0$ such that $[-\vep,\vep]\times[-\vep,\vep]\subset U$. In local coordinates the differential equation associated with the vector field $X$ is given by
\[x'=\frac{x}{V(x,y)},\; y'=-\frac{y}{V(x,y)}\qquad \text{ and }\quad t  \mapsto H(x(t),y(t))=x(t)y(t)\quad \text{ is constant.}\]
Therefore the forward semiorbit of any $\pm(x/\vep,\vep)$ with $x\in[-\vep^2,\vep^2]\setminus \{0\}$ leaves the square $[-\vep,\vep]\times[-\vep,\vep]$ at $\pm\sgn(x)(\vep, x/\vep)$. Moreover,
the time it takes to go through the square is
\begin{align*}
\tau(x)&=\int_0^{\tau(x)}dt=\int_0^{\tau(x)}\frac{V(x(t),x/x(t))x'(t)}{x(t)}dt
=\int_{|x|/\vep^2}^{1}\frac{V\big(\pm\vep\big(\sgn(x) s,\frac{|x|/\vep^2}{s}\big)\big)}{s}ds.
\end{align*}
To get the last equality we use the substitution $x(t)=\pm\sgn(x)\vep s$.
By Lemma A.1 in \cite{Fr-Ul}, $\tau(x)=-V(0,0)\log x+g(x)$, where $g:[-\vep^2,\vep^2]\to\R$ is of bounded variation. Let us consider the transversal curves $\gamma:[-\vep^2,\vep^2]\to M$
given by $\gamma(s)=\pm(s/\vep,\vep)$ or $\gamma(s)=\pm(\vep,s/\vep)$. Since $\eta$ in local coordinates is given by $\eta_{(x,y)}=y\,dx+x\,dy$, we
always have $\eta(d\gamma)=1$ so all of them are standard. As
\[K_{\omega,X}(p)=\frac{\det\operatorname{Hess}(H)(0,0)}{V^2(0,0)}=\frac{-1}{V^2(0,0)},\]
we have $V(0,0)=1/{\sqrt{-K_{\omega,X}(p)}}$. This completes the proof of $(iii)$ and hence of the Lemma.
\end{proof}

\begin{proof}[Proof of Corollary~\ref{cor:reduction}]
By Lemma  \ref{lemma:reduction} one can choose the special  representation of  $(\varphi_t)_{t \in \mathbb{R}}$  so that (by $(i)$, since  $(\varphi_t)_{t \in \mathbb{R}}$ has  two saddles)  $\pi$ is  symmetric on $d=2g+1=5$  and furthermore, since the saddles are both simple, by $(ii)$, $f\in \Log{\sqcup_{i=0}^{4}I_i}$.

 Suppose that the forward $(\varphi_t)_{t \in \mathbb{R}}$-orbit (or equivalently $(h_t)_{t\in\R}$-orbit) of $\beta_i$ ($0\leq i<5$) meets the saddle point $p$ before returning to $I$. Applying the involution $\iota$, we obtain that  the backward $(h_t)_{t\in\R}$-orbit of $S\beta_i=|I|-\beta_i$  meets the saddle point $\iota(p)$ before backward returning to $I$. Since $T$ transforms $(\beta_i,\beta_{i+1})$
on $(|I|-\beta_{i+1},|I|-\beta_i)$ by a translation, it follows that the forward $(h_t)_{t \in \mathbb{R}}$-orbit (or equivalently
$(\varphi_t)_{t\in\R}$-orbit) of $\beta_{i+1}$  meets the saddle point $\iota(p)$ before returning to $I$.
By $(iii)$ in Lemma~\ref{lemma:reduction} and the fact that $p$ and $\iota(p)$ are isomorphic, we have
\[C_i^+=\frac{1}{\sqrt{-K_{\omega,X}(p)}}=\frac{1}{\sqrt{-K_{\omega,X}(\iota(p))}}=C^-_{i+1}.\]
The same argument shows that if  the forward $(\varphi_t)_{t \in \mathbb{R}}$-orbit  of $\beta_i$ does not meet any saddle point  before returning to $I$ then $\beta_{i+1}$ satisfies the same property. Then $C^+_{i}=C^-_{i+1}=0$. By the proof of $(i)$ in Lemma~\ref{lemma:reduction}, $\beta_i$ and $\beta_{i+1}$ are the only two points satisfying this property, which completes the proof.
\end{proof}

\begin{rem}\label{rem:fullmeasure}
In the reduction described above of a locally Hamiltonian flow to a special flow over an IET $T$, one can see that the length of each interval $(\beta_i, \beta_{i+1})$ exchanged by $T$ coincide with one of the coordinates of $Per(\eta)$, where  we recall that $Per $ denotes the period map defined in \S\ref{sec:measureclass}.
{Thus,
for every subset $\mathcal{U}\subset \mathcal{U}_{min}$ of locally Hamiltonian flows,  the set $\{ Per (\eta), \eta \in \mathcal{U}\} $ has full Lebesgue measure as long as a full measure set of IET on $d$ intervals and fixed permutation (with respect to the Lebesgue measure on the lenghts of the intervals) appears in the base of special flows representations of flows in  $\mathcal{U}$.} 

{
Furthermore, to show that a property is \emph{typical} within  a subset $\mathcal{U}\subset \mathcal{U}_{min}$ of locally Hamiltonian flows (in the sense of \S\ref{sec:isomorphic} for $\mathcal{U}=\mathcal{K}$), it is sufficient to show that it holds for \emph{every} special flow representation of a flow in $\mathcal{U}$ over a full measure set of IETs in the base (in this way it can only fail only on the preimage via $Per$ of a zero Lebesgue mesure set). In particular,
to show that singularity of the spectrum holds for a typical flow in the isomorphic saddle locus $\mathcal{K}$ (recall Definition~\ref{def:isomorphiclocus}), it is enough to show that for almost every IET with a  symmetric permutation $\pi$ with $d=4$, every special flow with symmetric logarithmic singularities  $f\in \SymLog{\sqcup_{i=0}^{4}I_i}$ has singular spectrum.
}
\end{rem}

\subsection{Previous results on ergodic and spectral properties}\label{sec:history}
Let us briefly summarize the mixing and spectral results known for locally Hamiltonian flows and special flows over rotations and IETs.   Mixing properties of locally Hamiltonian flows turn out to depend crucially on the type of singularities (i.e.~fixed points) of the flow.

\noindent {\bf Flows with no singularities  or degenerate singularities.} If a smooth flow on a compact surface has no singularities, by Poincar{\'e}-Hopf theorem the surface has genus one and hence is a torus.  It is well known that smooth linear flows on the torus are typically not mixing, by KAM type results (see e.g.~\cite{Kolmogorov}). Furthermore, let us point out that Katok in \cite{Ka:int} showed that linear flows on translation surfaces (and more in general special flows over IETs -thus in particular over rotations- under a roof function of \emph{bounded variation}) are \emph{never}  mixing.

On the torus, one can introduce a fake singularity by adding a stopping point. This operation can drastically change the ergodic and spectral properties: as already mentioned in the introduction, Forni, Fayad and Kanigowski recently showed in \cite{FFK} that if the stopping point is sufficiently strong, the resulting flow  
  has countable Lebesgue spectrum. Stopping points can be thought as \emph{degenerate} saddles, with only two separatrices.
On a surface of any genus $g\geq 1$, the presence   of either a stopping point, or more in general, of a
 \emph{degenerate} critical points (which correspond to \emph{multi-saddles}, i.e.~saddles with $2n$ prongs, $n\neq 2$ integer, see Figure~\ref{multisaddle}) produce \emph{power-like} singularities in the
 special flow representation. Special flows over IETs with these type of singularities are mixing (for a full measure set of base transformations) by Kochergin's work \cite{Ko:mix}.


\noindent {\bf Flows with logarithmic singularities over rotations.} Singularities which are \emph{non-degenerate}, as we just saw (in the previous \S\ref{sec:reduction}) give rise to  special flows with \emph{logarithmic} singularities. In this case, mixing depends on the \emph{(a)symmetry} of the singularities. The first result on absence of mixing for special flows with \emph{symmetric} logarithmic singularities over \emph{rotations}
is due to Kochergin \cite{Ko:abs} (see also \cite{Ko:07} where the result was proved for all irrational frequencies). If the roof has an asymmetric logarithmic singularities, instead, mixing is typical, as it was proved by Sinai-Khanin for a full measure set of rotations numbers  (see also further works by Kochergin \cite{Ko:mix, Ko:03, Ko:04, Ko:04'}). These flows are also known as Arnold flows since Sinai-Khanin result \cite{SK:mix} proved Arnold's conjecture \cite{Ar:top} on mixing of typical locally Hamiltonian flows with a saddle point on the torus).

Stronger mixing and spectral properties were later shown for flows over rotations. First, as already mentioned in the introduction, Fr\k{a}czek and Lema\'nczyk in \cite{Fr-Le0}  showed that flows under a symmetric logarithm over a full measure set of rotation numbers are disjoint from all mixing flows and have singular spectrum.
Fayad and Kanigowski recently proved in \cite{FK} that Arnold flows
 (as well as some Kochergin flows -i.e.\ flows under roofs with power-type singularities- over rotations)
 are mixing of all orders. Kanigowski, Lema\'nczyk and Ulcigrai~proved some disjointness properties (in particular disjointness of rescalings) for typical Arnold flows, \cite{KLU}.

\noindent {\bf Flows with logarithmic singularities over IETs.}
Fewer results are available for flows with logarithmic singularities over IETs. A simple mechanism that shows that  weak mixing (or, equivalently, continuity of the spectrum) holds typically  as long as there is a logarithmic singularity (recall that weak mixing is also know to hold for typical translation flows by \cite{AF:wea}). Mixing again depends on the symmetry. Scheglov proved in \cite{Sch:abs} that typical minimal locally Hamiltonian flows with isomorphic simple saddles in $g=2$ are not mixing.
Ulcigrai showed in \cite{Ul:abs} that, for typical IETs, flows with symmetric singularities are not mixing (thus, in the open set  $\mathscr{U}_{min}$ of locally Hamiltonian flows, the typical flow, which is minimal and ergodic, is weak mixing but not mixing). Nevertheless, the existence of a mixing flow under a symmetric roof function (smoothly realized by a minimal, locally Hamiltonian flow with only simple saddles on a surface of genus $g=5$) was proved by Chaika~and Wright in \cite{CW}.
	
	On the other hand, generalizing  Sinai-Khanin result \cite{SK:mix}, Ulcigrai showed in  \cite{Ul:mix} that flows over IETs with one asymmetric logaritmic singularity are mixing for almost every IET. These result was recently generalized by Ravotti in~\cite{Ra:mix} to any number of singularities (thus showing that 
in presence of saddle loops homologous to zero the typical locally Hamiltonian flow with non-degenerate zeros has mixing minimal components). Recent strengthenings of the mixing property were also proved:
Ravotti  in~\cite{Ra:mix} also proved quantitative (subpolynomial) bounds on the speed of mixing for smooth observables. Finally, in
 \cite{KKU} it was shown that for a full measure (sub)set of IETs, flows with asymmetric logarithmic singularities are mixing of all orders (and thus that mixing implies mixing of all orders for typical smooth area-preserving flows  any genus.)

\section{A criterion  for singularity in special flows}\label{sec:singcrit}
 In this section we present a {sufficient} condition (originally formulated in \cite{Fr-Le0,Fr-Le1} in a slightly less general form) which guarantees that $(T^f_t)_{t\in\R}$ has singular spectrum.  We first recall some basic spectral theory.

\subsection{Spectral notions}\label{sec:spectral}

The spectrum and the spectral properties of a measure-preserving flow $(T_t)_{t \in \mathbb{R}}$ acting on a probability Borel space $(X,\mathcal{B},\mu)$ are defined in terms of the Koopman (unitary) operators associated to $(T_t)_{t \in \mathbb{R}}$. Let us recall that, for every $t\in\R$, the \emph{Koopman} \emph{operator} associated to the automorphism  $T_t$, which, abusing the notation,  we will denote also by $T_t$, is the operator
\[T_t:L^2(X,\mu)\to L^2(X,\mu)\quad \text{  given by \ \ }T_t(f)=f\circ T_t.\]
To every $g\in L^2(X,\mu)$ one can associate  a  spectral measure  denoted by  $\sigma_g$, i.e.\  the unique finite Borel measure on $\R$ such that
\[\langle g\circ T_{t},g\rangle=\int_\R e^{its}\,d\sigma_g(s)\quad\text{for every}\quad t\in \R.\]
The spectrum of $(T_t)_{t \in \mathbb{T}}$ is (\emph{purely}) \emph{singular} iff for every $g\in L^2(X,\mu)$ the spectral measure $\sigma_g$ is singular with respect to the Lebesgue measure on $\R$.

Let us denote by $\R(g)\subset  L^2(X,\mu)$ the \emph{cyclic subspace} generated by $g$ which is given by
\[\R(g):=\overline{\spa}\{T_t(g):t\in\R\}\subset  L^2(X,\mu)\]
By the spectral theorem (see e.g.~\cite{CFS:erg}) the Koopman $\R$-representation $(T^f_{t})_{t\in\R}$ restricted to $\R(g)$
is unitarily isomorphic to the $\R$-representation $(V_t)_{t\in\R}$ on $L^2(\R,\sigma_g)$ given by
$V_t(h)(s)=e^{its}h(s)$.

Finally, let us recall the notion of \emph{integral operator}. For every probability Borel measure $P$ on $\R$ denote by $\int_\R T_t\,dP(t):L^2(X,\mu)\to L^2(X,\mu)$
the operator such that
\[\langle\int_\R T_t\,dP(t)(g_1),g_2\rangle=\int_\R\langle T_t(g_1),g_2\rangle\,dP(t)\]
for all $g_1,g_2\in L^2(X,\mu)$.

\subsection{The singularity criterion}\label{sec:sc}
We will now state the singularity criterion for special flows $(T^f_t)$, which are based on rigidity of the base $T:X\to X$ combined with exponential tails of the Birkhoff sums for the roof function $f:X\to\R_{>0}$. Let us first recall the notion of \emph{rigidity}.

\begin{defn}[Rigidity] An automorphism $T$ of a probability Borel space $(X,\mathcal{B},\mu)$ is called \emph{rigid} if there exists an increasing sequence $(h_n)_{n\in\N}$
of natural numbers  such that
\[\lim_{n\to+\infty}\mu(A\triangle T^{h_n}A)=0\quad\text{for every}\quad A\in\mathcal{B}.\]
The sequence $(h_n)_{n\in\N}$ is then a \emph{rigidity sequence} for $T$.
\end{defn}
For every $B\in\mathcal B$ with $\mu(B)>0$, we  denote by $\mu_B$ the \emph{conditional measure} given by  $\mu_B(A)=\mu(A|B)=\mu(A\cap B)/\mu(B)$ for every measurable $A$.

\begin{thm}[Singularity Criterion via rigidity and exponential tails]\label{thm:singcrit}
Let $f:X\to\R_{>0}$ be an integrable roof function with $\inf_{x\in X}f(x)>0$.
Suppose that there exist a {rigidity sequence} $(h_n)_{n\in\N}$ for $T$, a sequence $(C_n)_{n\in\N}$  of Borel sets with $\mu(C_n)\to 1$ as $n\to+\infty$,
 and a  sequence of real numbers $(c_n)_{n\in\N}$ (centralizing constants)  such that $(S_{h_n}(f)(x)-c_n)_{n\in \mathbb{N}}$ has exponential tails, i.e.~
 there exists two positive constants $C$ and $b$ such that
\begin{equation}\label{eq:expdecay}
\mu(\{x\in C_n:|S_{h_n}(f)(x)-c_n|\geq t\})\leq Ce^{-bt}\text{ for all }t\geq0\text{ and }n\in\N.
\end{equation}
Then the flow $(T^f_{t})_{t\in\R}$ has singular spectrum.
\end{thm}
\begin{rem}\label{rk:tightness}
The \emph{exponential tails} assumption, i.e.~\eqref{eq:expdecay}, along rigidity sets implies in particular that the sequence of centralized Birkhoff sums  $(S_{h_n}(f)(x)-c_n)_{n\in \mathbb{N}}$ is \emph{tight}. Tightness of Birkhoff sums along (partial) rigidity subsequences of the base is at the heart of many criteria for absence of mixing, starting from  Katok~\cite{Ka:int} and Kogergin~\cite{Ko:abs} seminal works. Theorem~\ref{thm:singcrit} can hence be seen as  considerable \emph{strengthening} of this approach to absence of mixing and  shows that tightness and rigidity (of the base), with the additional information of exponential tails, is sufficient to show singularity of the spectrum. Contrary to proofs of absence of mixing, though, it is crucial for the spectral conclusion that the rigidity here is \emph{global}, i.e.\ the measure of the sets $C_n$ tends to $1$.
\end{rem}

We conclude this section with the proof of the criterion. In the proof we will use the following result from \cite{Fr-Le2}, which is a version of Prokhorov weak compactness of tight sequences along rigidity sets.

\begin{prop}[Theorem~6 in \cite{Fr-Le2}]\label{prop:operconv}
Suppose that there exist a rigidity sequence $(h_n)_{n\in\N}$ for $T$, a sequence $(C_n)_{n\in\N}$  of Borel sets with $\mu(C_n)\to 1$ as $n\to+\infty$
and a sequence of real numbers $(c_n)_{n\in\N}$ such that
\begin{itemize}
\item[$(i)$] the sequence $(\int_{C_n}|f_n|^2\,d\mu|)_{n\in \N}$ is bounded, where $f_n:=S_{h_n}(f)-c_n$;
\item[$(ii)$] there exists a probability distribution $P$ on $\R$ such that $(f_n|{C_n})_*(\mu_{C_n})\to P$ weakly.
\end{itemize}
Then, passing to a subsequence if necessary, we have
\[T^f_{c_n}\to \int_\R T^f_{-t}\,dP(t) \text{ in the weak operator topology.}\]
\end{prop}

\begin{proof}[Proof of Theorem~\ref{thm:singcrit}]
Suppose that, contrary to our claim, the spectral measure $\sigma_g$ (see \S\ref{sec:spectral} for the definition) is absolutely continuous for some non-zero $g\in L^2(X^f,\mu^f)$.
Then, by the Riemann-Lebesgue lemma,
\begin{equation}\label{eq:RL}
V_t\to 0\text{ in the weak operator topology on }L^2(\R,\sigma_g)\text{ as }|t|\to+\infty.
\end{equation}
Note that, by \eqref{eq:expdecay}, we have
\begin{align*}
\int_{C_n}|S_{h_n}(f)(x)-c_n|^2\,d\mu(x)&\leq\sum_{m=0}^\infty\int_{\{x\in C_n:m\leq|S_{h_n}(f)(x)-c_n|<m+1\}}(m+1)^2\,d\mu(x)\\
&\leq\sum_{m=0}^\infty C(m+1)^2e^{-bm}.
\end{align*}
Hence, the condition ($i$) in Proposition~\ref{prop:operconv} is satisfied. Moreover, also by \eqref{eq:expdecay}, the sequence of probability Borel measures
$((f_n|{C_n})_*(\mu_{C_n}))_{n\in\N}$ on $\R$ is uniformly tight. Therefore, passing to a subsequence, we have $(f_n|{C_n})_*(\mu_{C_n})\to P$ weakly for some probability measure $P$ and the condition $(ii)$ in Proposition~\ref{prop:operconv} is satisfied.
Moreover,  $P$ has exponentially decaying tails, i.e.\ there exist $C,b>0$ such that
\begin{equation*}\label{eq:tails}
P((-\infty,-t)\cup(t,+\infty))\leq Ce^{-bt}\quad\text{for all}\quad t\geq 0.
\end{equation*}
In view of Proposition~\ref{prop:operconv}, we have
\begin{equation*}
T^f_{c_n}\to \int_\R T^f_{-t}\,dP(t)\text{ in the weak operator topology}.
\end{equation*}
Restricting this convergence to the invariant subspace $\R(g)$ and passing to $L^2(\R,\sigma_g)$, this gives
\begin{equation*}
V_{c_n}\to \int_\R V_{-t}\,dP(t)\text{ in the weak operator topology on }L^2(\R,\sigma_g).
\end{equation*}
In view of \eqref{eq:RL}, it follows that for all $h_1,h_2\in L^2(\R,\sigma_g)$ we have
\begin{align*}
0&=\langle \int_\R T^f_{-t}\,dP(t)(h_1),h_2\rangle=\int_{\R}\int_\R e^{-its}h_1(s)\overline{h_2}(s)\,d\sigma_g(s)\,dP(t)
=\int_{\R}\widehat{P}(s)h_1(s)\overline{h_2}(s)\,d\sigma_g(s),
\end{align*}
where $\widehat{P}$ is the Fourier transform of the measure $P$. Therefore (since $h_1$ and $h_2$ are arbitrary), $\widehat{P}(s)=0$ for $\sigma_g$ a.e.\ $s\in\R$.
On the other hand, as $P$ has exponentially decaying tails, its Fourier transform $\widehat{P}$ is an analytic function on $\R$.
It follows that $\widehat{P}\equiv 0$, contrary to non-triviality of the measure $\sigma_g$. This completes the proof.
\end{proof}

\begin{rem}
In fact, the proof of Theorem~\ref{thm:singcrit} also gives spectral disjointness of $T_f$ from all mixing flows.
\end{rem}

\section{Logarithmic singularities, symmetries and exponential tails}
In this section we will verify that, in the settings of Theorem~\ref{thm:ss_sf} the assumptions of the singularity criterion given by Theorem~\ref{thm:singcrit} hold.

\subsection{Birkhoff sums of functions with logarithmic singularities}\label{sec:BS}
In this section we present some estimates on Birkhoff sums of functions with  logarithmic singularities which will be used to prove the exponential tails assumption. The results are all elementary, essentially based on the mean value theorem. The precise form of the estimates  though  allows us to have a detailed control of the behavior of the tails, i.e.\ the way Birkhoff sums explode due to the presence of singularities.
\smallskip

The following general Lemma shows that  control  of the exponential tails  can be deduced from upper bounds on the second derivative. It will be applied below to $g = S_{h_n}(f)$ (Birhoff sums of $f$ along rigidity times $h_n$).

\begin{lem}[Exponential tails control]\label{lem:ad0}
Suppose that $g:(a,b)\to\R$ is a $C^2$-function such that
\[|g''(x)|\leq\frac{C}{(x-a)^2}+\frac{C}{(b-x)^2}\text{ for every }x\in(a,b).\]
Let $y_0=\tfrac{a+b}{2}$ and assume that there exists $x_0\in(a,b)$ be such that $g'(x_0)=0$.  Then
\begin{equation}\label{eq:dif}
|g(x)-g(y_0)|\leq C\Big(-\log\frac{x-a}{b-a}-\log\frac{b-x}{b-a}+\frac{b-a}{x_0-a}+\frac{b-a}{b-x_0}\Big)\text{ for every }x\in(a,b).
\end{equation}
In particular, for every $t\geq 0$ we have
\begin{equation}\label{eq:expd}
\frac{\lambda(\{x\in(a,b):|g(x)-g(y_0)|\geq t\})}{b-a}\leq 2\sqrt{K}e^{-t/2C},
\end{equation}
where {$K:=\tfrac{1}{4}e^{\tfrac{b-a}{x_0-a}+\tfrac{b-a}{b-x_0}}$}.
\end{lem}
\begin{rem}\label{rk:exptails}
Notice that $K$  depends only on the point $x_0$ where $g'(x_0)=0$, so that in order for Lemma~\ref{lem:ad0} to imply exponential tails estimates for  a sequence of functions with a  \emph{uniform} constant $K$ it is essential to control $x_0$ and in particular its distance from the endpoints $a,b$. In \S\ref{sec:symmetry}, Lemma~\ref{lem:ad0} will be applied to  the function $g=S_{h_n}(f)$ on one of the maximal intervals $(a,b)$ in which it is continuous. The assumption that there exists $x_0$ such that $S_{h_n}(f')(x_0)=g'(x_0)=0$ follows easily from the fact that $S_{h_n}(f)$ explodes at the endpoints of  $(a,b)$ (and hence has a minimum). On the other hand, the \emph{location} of $x_0 \in (a,b)$ is not in general easy to control and we will crucially use arguments which exploit the hyperelliptic symmetry to control $x_0$.
\end{rem}

\begin{proof}
By assumption, for every $x\in(a,b)$ we have
\begin{align*}
|g'(x)|&=|g'(x)-g'(x_0)|\leq\Big|\int_{x_0}^x\Big(\frac{C}{(t-a)^2}+\frac{C}{(b-t)^2}\Big)\,dt\Big|
\leq C\Big(\Big|\frac{1}{b-x}-\frac{1}{x-a}\Big|+\frac{1}{x_0-a}+\frac{1}{b-x_0}\Big).
\end{align*}
It follows that, for every $x\in(a,b)$ we have
\begin{align*}
|g(x)-g(y_0)|&\leq C\Big|\int_{y_0}^x\Big(\Big|\frac{1}{b-t}-\frac{1}{t-a}\Big|+\frac{1}{x_0-a}+\frac{1}{b-x_0}\Big)\,dt\Big|\\
&\leq C\Big(\int_{y_0}^x\Big(\frac{1}{b-t}-\frac{1}{t-a}\Big)\,dt+\frac{|x-y_0|}{x_0-a}+\frac{|x-y_0|}{b-x_0}\Big)
\\
&\leq C\Big(-\log\frac{x-a}{(b-a)/2}-\log\frac{b-x}{(b-a)/2}+\frac{b-a}{x_0-a}+\frac{b-a}{b-x_0}\Big),
\end{align*}
since $|x-y_0|\leq b-a$ and recalling that $y_0-a=b-y_0=(b-a)/2$.
This gives \eqref{eq:dif}.
Moreover, by \eqref{eq:dif}, if $|g(x)-g(y_0)|\geq t$ then
\[\frac{(x-a)(b-x)}{(b-a)^2}\leq Ke^{-t/C}.\]
Note also that for any $u>0$ we have
\[\frac{\lambda(\{x\in(a,b):\tfrac{(x-a)(b-x)}{(b-a)^2}\leq u\})}{b-a}=\lambda(\{x\in(0,1):x(1-x)\leq u\})\leq 2\sqrt{u},\]
which gives \eqref{eq:expd}.
\end{proof}

The following lemma (Lemma \ref{lem:ad1}) provides an upper bound on the  second derivative $S_{h_n}(f'')$
which is exactly of the form needed to verify the assumption of Lemma \ref{lem:ad0} and prove exponential tails.

\begin{defn}[Rohlin tower by intervals]\label{def:tower}
Let $T:I\to I$ be an IET with $|I|\leq 1$. Given an interval $J_n:=(a_n,b_n)\subset I$ and an integer $h_n\in \N$ we say that the union $\cC :=\bigcup_{i=0}^{h_n-1}T^i J_n$ is a (Rohlin) \emph{tower by intervals} of \emph{base} $J_n$ of \emph{height} $h_n$ if and only if the images $T^i J_n, 0\leq i < h_n$ are pairwise disjoint intervals.
\end{defn}
We remark that, since each $T^i J_n$ is by assumption an interval (i.e.~it was not split by discontinuities of $T$), $J_n$ is an interval of \emph{continuity} for $T^i$ for every $0\leq i <h_{n}$.

\begin{lem}[Second derivative upper bounds]\label{lem:ad1}
Consider a function $f\in\pLog{\sqcup_{i=0}^{d-1}I_i}$  of the form
\[f(x)=\sum_{0\leq i<d}\big(-C_i^+\log(x-\beta_i)-C_{i+1}^-\log(\beta_{i+1}-x)\big)\chi_{(\beta_i,\beta_{i+1})}(x).
\]
Assume that $x\in J_n $ where $J_n:=(a_n,b_n)\subset I $ is the base of Rohlin tower by intervals of height $h_n\in \N$.
Then
\begin{equation}\label{eq:Sf''}
|S_{h_n}(f'')(x)|\leq \frac{\pi^2}{6}\Big(\frac{C^+}{(x-a_n)^2}+\frac{C^-}{(b_n-x)^2}\Big)
\end{equation}
with $C^+=\sum_{i=0}^{d-1}C_i^+$ and $C^-=\sum_{i=1}^dC_i^-$. Moreover, if $a_n<x<x'<b_n$ then for every $0\leq h<h_n$, we have
{\begin{align}\label{eq:Sf}
\begin{aligned}
|S_{h}(f)(x)-S_{h}(f)(x')|&\leq C^+\Big(\frac{x'-x}{x-a_n}+\frac{x'-x}{b_n-a_n}\Big(1+\log\frac{1}{b_n-a_n}\Big)\Big)\\
&\quad+ C^-\Big(\frac{x'-x}{b_n-x'}+\frac{x'-x}{b_n-a_n}\Big(1+\log\frac{1}{b_n-a_n}\Big)\Big).
\end{aligned}
\end{align}}
\end{lem}
\begin{rem}
{We stress that \eqref{eq:Sf''} holds only for \emph{Birkhoff sums along a tower}, i.e.~ a point $x$ is in the base of
the Rohlin tower
and the Birkhoff sum  goes up to height $h_n$ of the tower, 
while  \eqref{eq:Sf} holds for points in the base $J_n$ and for any \emph{intermediate} time $0\leq h< h_n$.} Under the assumption that $x$ is not close to the endpoints $a_n,b_n$ of $J_n$, namely if $ x- a_n\geq c(b_n-a_n) $ and $ b_n-x \geq c(b_n-a_n) $ for some $0<c<1$, \eqref{eq:Sf''} together with Lemma~\ref{lem:ad0} provide  a uniform bound (independent on $n$) for the difference in \eqref{eq:Sf}, while  \eqref{eq:Sf} provides only an upper bound of order $C \log (b_n-a_n)^{-1}$. This is  a well-known upper bound for functions with \emph{asymmetric} logarithmic singularities (see e.g.~\cite{SK:mix, Ul:mix, Ra:mix}, where it is shown that $S_{h}(f')(x)$ grows as $h_n \log h_n$). While Birkhoff sums along a (large) tower are well distributed, incomplete sums can indeed be very unbalanced and only satisfy estimates associated to an asymmetric roof.
\end{rem}
\begin{proof}
Notice first that it is enough to prove \eqref{eq:Sf''} and \eqref{eq:Sf} in the special cases when
\begin{align*}&f=f_i^+:=-\log(x-\beta_i)\chi_{(\beta_i,\beta_{i+1})}(x) \quad\text{(and $C^+=1, C^- =0$) and}\\
&f=f^-_i:=-\log(\beta_{i+1}-x)\chi_{(\beta_i,\beta_{i+1})}(x)\quad\text{(and $C^+=0, C^- =1$).}
\end{align*}
Indeed, taking the linear combination $\sum_{i=0}^{d-1}C_i^+f_i^+ +C_{i+1}^- f_{i}^-$ then yields the general form of the result.
Since the reasoning is analogous for functions of the form  $f_i^+$  or $f_i^-$ we will only do the computations  for $f=f_i^+$.

{For any $x\in J_n$ choose  $0\leq j<h_n$ such that the iterate
$T^jx$ is the closest to $\beta_i$ among all iterates $T^kx$, $0\leq k<h_n$ belonging to the interval $(\beta_i, \beta_d)$.
Then
\[T^jx-\beta_i\geq x-a_n.\]
Notice  that $\inf \{ \vert T^i (x)- T^j(x)\vert, 0 \leq i\neq j <h\}\geq b_n-a_n$. Indeed, each point of the orbit $T^h (x)$ for $0\leq h < h_n$ belongs to one of the disjoint intervals $\{ T^h J_n, 0\leq h < h_n\}$, each of which (since $T^h$ is an isometry on $J_n$) has length $b_n-a_n$. Thus,
\begin{align*}
|S_{h_n}(f'')(x)|&\leq \sum_{0\leq l<h_n}\frac{1}{(T^jx-\beta_i+l(b_n-a_n))^2}
\leq \frac{1}{(x-a_n)^2}\sum_{0\leq l<h_n}\frac{1}{(1+l\tfrac{b_n-a_n}{x-a_n})^2}\\
&\leq \frac{1}{(x-a_n)^2}\sum_{l\geq 1}\frac{1}{l^2}=\frac{\pi^2}{6}\frac{1}{(x-a_n)^2}.
\end{align*}}
This gives \eqref{eq:Sf''}.

Suppose now that $a_n<x<x'<b_n$ and $f=f_i^+$. Let $\delta:=x'-x$ and $\ep:=x-a_n$. Then 
for any $0\leq h<h_n$ (noticing that since the area of a tower is less than one,  $h_n(b_n-a_n) \leq 1$),
{\begin{align*}
|S_h(f)(x')-S_h(f)(x)|&\leq \sum_{\substack{0\leq k<h_n\\T^kx>\beta_i}}\log\frac{(T^kx-\beta_i)+\delta}{T^kx-\beta_i}\leq
\sum_{\substack{0\leq k<h_n\\T^kx>\beta_i}}\frac{\delta}{T^kx-\beta_i}\\
& \leq \sum_{0\leq l<h_n}\frac{\delta}{(T^jx-\beta_i)+l(b_n-a_n)}\leq \sum_{0\leq l<h_n}\frac{\delta}{\ep+l(b_n-a_n)}\\
& \leq\frac{\delta}{\ep}+\sum_{1\leq l<h_n}\frac{\delta}{l(b_n-a_n)}
\leq\frac{\delta}{\ep}+\frac{\delta}{b_n-a_n}(1+\log h_n)\\
&\leq \frac{\delta}{\ep}+\frac{\delta}{b_n-a_n}\Big(1+\log\frac{1}{b_n-a_n}\Big).
\end{align*}}
In virtue of the initial remark, this concludes the proof of  \eqref{eq:Sf}.
\end{proof}

From Lemmas \ref{lem:ad0} and \ref{lem:ad1}, we can deduce exponential tails as long as we can control the location of a zero of $S_{h_n}(f')$.
\begin{cor}\label{cor:exptails}
Let $y_n=\tfrac{a_n+b_n}{2}$. Assume that there exists $0<c<1/2$ such that for every $n\geq 1$ we have $S_{h_n}(f')(x_n)=0$ for some $x_n\in [a_n+c(b_n-a_n),b_n-c(b_n-a_n)]$.
Then for every $n\geq 1$ and $t\geq 0$ we have
{\[
\frac{1}{b_n-a_n}\lambda(\{x\in J_n:|S_{h_n}(f)(x)-S_{h_n}(f)(y_n)|\geq t\})\leq e^{1/c}e^{-t/2C},
\]}
where $C=\frac{\pi^2}{6}\max\left\{\sum_{i=0}^{d-1}C_i^+,\sum_{i=1}^dC_i^-\right\}$.
\end{cor}
\begin{proof} This is a consequence of Lemma~\ref{lem:ad0} (applied to $g=S_{h_n}(f)$)  and Lemma~\ref{lem:ad1}: it is enough to notice that  since $x_n\in [a_n+c(b_n-a_n),b_n-c(b_n-a_n)]$, the constant $K$  (see Lemma \ref{lem:ad0}) is globally bounded (in terms of $c>0$).  Indeed,
{\[K=\tfrac{1}{4}e^{\tfrac{b_n-a_n}{x_n-a_n}+\tfrac{b_n-a_n}{b_n-x_n}}
\leq \frac{e^{\tfrac{2}{c}}}{4},\]}
which completes the proof.
\end{proof}

\subsection{Hyperelliptic symmetry and cancellations.}\label{sec:symmetry}
In this section we show that the symmetries of an IET with a symmetric $\pi$ and a  roof function with pure symmetric logarithmic singularities allow us to determine critical points  of $S_{h_n} (f)$. 

\begin{rem}\label{rem:invertcond}
Let $(T_t)_{t\in\R}$ be a measure-preserving flow on $(X,\mathcal{B},\mu)$ and let $S$ be a measure-preserving involution
such that
\begin{equation}\label{invert:flow}
T_t\circ S=S\circ T_{-t}\;\text{ for every }\;t\in\R.
\end{equation}
Let $I\subset X$ be global transversal for the flow $(T_t)_{t\in\R}$ such that $S(I)=I$. Let $T_I:I\to I$ is the first return map to $I$
and $f:I\to\R_{>0}$ be the first return time. Then, it is easy to check
that,
\begin{equation}\label{invert:special}
T_I\circ S=S\circ T_I^{-1}\;\text{ and }\;f\circ T_I^{-1}\circ S=f.
\end{equation}
In fact, the conditions \eqref{invert:flow} and \eqref{invert:special} are in a sense equivalent.
Indeed, if $S:I\to I$ is an involution satisfying \eqref{invert:special}, then it has an
extension to the involution $S^f:I^f\to I^f$ given by
\[S^f(x,y):=(Sx,-y)\;\text{ for all }\;(x,y)\in I^f.\]
Then \eqref{invert:special} implies \eqref{invert:flow} for the special flow $T_I^f$.
\end{rem}
Assume throughout this section that $T:I\to I$ 
is an IET associated with the \emph{symmetric} permutation $\pi(i)=d-1-i$ for $0\leq i < d$. Recall that $\pSymLog{\sqcup_{i=0}^{d-1}I_i}$ denotes functions with \emph{pure symmetric logarithmic singularities} (see Definition~\ref{def:SymLog}). Such IETs and functions enjoy the following symmetries.

\begin{lem}[Symmetries]\label{lem:symmetries}
Let $T$ be an IET with  a symmetric $\pi$ and endpoints $0=\beta_0< \dots <\beta_d=|I|$ and assume $f \in \pSymLog{\sqcup_{i=0}^{d-1}I_i}$. Then, if $S:I\to I$ denotes the involution $S(x)=|I|-x$,
\begin{itemize}
\item[$(SB)$] $T\circ S=S\circ T^{-1}$;
\item[$(SR)$]  $f'\circ T^{-1}\circ S=-f'$.
\end{itemize}
\end{lem}
\noindent [$SB$  stands for Symmetries of the Base and $SR$ for Symmetries of the Roof.]  Notice that the relation $(SB)$ is the same that appeared in the proof of Lemma~\ref{lemma:reduction}, see \eqref{IETsymmetry}.
\begin{proof}
Since $\pi$ is symmetric, $T$ maps $[\beta_i,\beta_{i+1})$ linearly on $[|I|-\beta_{i+1},|I|-\beta_{i})$, i.e.\
\[Tx=x+|I|-\beta_i-\beta_{i+1}\text{ for all }x\in [\beta_i,\beta_{i+1}).\]
Thus, one can verify directly that $(SB)$ holds.
\smallskip
We claim that  a measurable function $\phi:I\to\R\cup\{\pm\infty\}$  satisfies $\phi\circ T^{-1}\circ S=-\phi$ if and only if
\begin{equation}\label{eq:antisym}
\phi(x)=-\phi(\beta_i+\beta_{i+1}-x)\text{ for all }x\in (\beta_i,\beta_{i+1})\text{ and any }0\leq i< d.
\end{equation}
Indeed, for every $x\in (\beta_i,\beta_{i+1})$ we have
\[\phi( T^{-1} (Sx))=\phi( T^{-1} (|I|-x))=\phi(|I|-x -(|I|-\beta_i-\beta_{i+1}))=\phi(\beta_i+\beta_{i+1}-x).\]
This gives our claim.

Since $f \in \pSymLog{\sqcup_{i=0}^{d-1}I_i}$,
\begin{equation}\label{eq:logder}
f'(x)=-\frac{C^+_i}{x-\beta_i}+\frac{C^-_{i+1}}{\beta_{i+1}-x}\text{ if }x\in(\beta_i,\beta_{i+1})\qquad \text{ for }0\leq i< d,
\end{equation}
where $C^+_i=C^-_{i+1}=C$ for $0\leq i< d$.  Hence one sees that $\phi:= f'$  satisfies \eqref{eq:antisym} and hence $(SR)$ holds.
\end{proof}

\begin{rem}\label{rem:symcond}
By the proof of Lemma~\ref{lem:symmetries}, we also have that for every $f \in \pLog{\sqcup_{i=0}^{d-1}I_i}$ the symmetry condition ($C^+_i=C^-_{i+1}$ for $0\leq i< d$) is equivalent to $f\circ T^{-1}\circ S=f$.
\end{rem}

The relations in Lemma~\ref{lem:symmetries} automatically imply the symmetry of Birkhoff sums stated in Lemma~\ref{lem:cancellations} below and hence allows us to locate $x_0$ such that $S_n( f')(x_0)=0$ (see Corollary~\ref{cor:x0})

\begin{lem}[Cancellations]\label{lem:cancellations}
Suppose that $T$ and $S$ are measure-preserving automorphisms of a probability Borel space $(X,\mathcal{B},\mu)$
such that $T\circ S=S\circ T^{-1}$ and $S$ is idempotent ($S^2=Id$). Assume that $\phi:X\to\R$ is a measurable map
with $\phi\circ T^{-1}\circ S=-\phi$. Then for every $n\in\N$ and $x\in X$ we have
$$
S_n(\phi)(T^{-n}(Sx))=-S_n(\phi)(x).
$$
In particular, if $x_0\in X$ is a fixed point of $S$ ($Sx_0=x_0$), then for every $n\in\N$ we have
$$
S_{n}(\phi)(T^{-n}x_0)=-S_{n}(\phi)(x_0).
$$
\end{lem}
\begin{proof}
The first part follows simply by the chain of equalities
\begin{align*}
S_n(\phi)(T^{-n}(Sx))&=\sum_{0\leq i<n}\phi(T^{i-n}(Sx))=\sum_{0\leq i<n}\phi(T^{-1-i}(Sx))\\
&=\sum_{0\leq i<n}\phi(T^{-1}(S(T^ix)))
=-\sum_{0\leq i<n}\phi(T^ix)=-S_n(\phi)(x).
\end{align*}
The second part is also immediate.
\end{proof}

Combining  Lemmas \ref{lem:symmetries} and \ref{lem:cancellations}  we have the following Corollary.
\begin{cor}[Cancellations]\label{cor:x0}
Let $T$ be an IET with  a symmetric $\pi$ and endpoints $0=\beta_0< \dots <\beta_d=|I|$ and $f \in \pSymLog{\sqcup_{i=0}^{d-1}I_i}$. For every $n\in \N$, we have  
\begin{equation}\label{eq:asymx0}
S_n(f')(T^{-n}x_0)=-S_n(f')(x_0)\ \ \text{for} \ \ x_0=|I|/2.
\end{equation}
Moreover, if $(a,b)$ is an interval on which $S_{n}(f')$ is continuous and both $x_0$ and $T^{-n}x_0$ belong to $(a,b)$,
 it follows that there exists
\[x_n \in  (a,b)\quad\text{such that}\quad S_{n}(f')(x_n)=0.\]
\end{cor}
\begin{proof}
By Lemma~\ref{lem:symmetries}, the assumptions of Lemma~\ref{lem:cancellations} hold for $T: I \to I$, $f': I \to \R$ and $S:I \to I$ given by $S(x)=|I|-x$. Since $x_0=|I|/2$ is the (unique) fixed point of the involution $S$, the first part follows immediately from  Lemma~\ref{lem:cancellations}.

We claim that the second part is simply an application of the intermediate value theorem. Indeed,
first note that since $S_{n}(f')$ is by assumption continuous on $(a,b)$ and $f$ has pure logarithmic singularities, it is actually smooth.
By the first part of the Corollary,
$S_{n}(f')(T^{-n}x_0)=-S_{n}(f')(x_0)$ and by assumption both $x_0$ and $T^{-n}x_0$ belong to $(a,b)$, so
 $S_n(f')$ changes sign on $(a,b)$ and hence must have a zero.
\end{proof}

\subsection{Good rigidity and exponential tails.}\label{sec:exptails}
The last ingredient we need to verify the assumptions of  the singularity criterion are
rigidity  sequences given by Rohlin towers with good recurrence on the base (in the sense of Definition~\ref{def:IETrecurrence} below).
Recall that Rohlin towers by intervals were defined at the beginning of \S\ref{sec:BS} (see Definition~\ref{def:tower}).

\begin{defn}[Good rigidity]\label{def:IETrecurrence}
We say that $T: I\to I$ admits a \emph{good rigidity sequence}  
if there exists a sequence of Rohlin towers by intervals $\cC_n\subset I$ of base $J_n=[a_n,b_n]$ and height  $h_n$ such that
\begin{itemize}
\item[$(GR1)$] $\lambda(\cC_n)\to |I|$
\end{itemize}
and, if we define  $q_n =\frac{1}{b_n-a_n}$ and $\ep_n:= \frac{1}{q_n\log q_n}$,
\begin{itemize}
\item[$(GR2)$]
the tower $\cC_n$ is $\ep_n$-\emph{rigid}, that is, for every $x\in \cC_n$, we have
\[
|T^{h_n}x-x|\leq\ep_n:= \frac{1}{q_n\log q_n}.
\]
\end{itemize}
\end{defn}
This  \emph{good} form of \emph{recurrence} (which will be deduced in \S\ref{sec:cylinder} by the abundance of directions well approximated by cylinders, see Lemma~\ref{lem:rigidity_from_cyl}) provides the final key ingredient to the proof of singularity of the spectrum for special flows with symmetric logarithmic singularities.

\begin{prop}[Singularity for symmetric logarithmic flows from good rigidity]\label{prop:an}
Let $T$ be an IET with a symmetric permutation $\pi$ and endpoints $0=\beta_0< \dots < \beta_d=|I|$ 
and assume that $f \in \SymLog{\sqcup_{0\leq i<d}  I_i}$ has symmetric logarithmic singularities.

 If $T$  admits a \emph{good rigidity sequence} of Rohlin towers  $\cC_n\subset I$ with bases $J_n=[a_n,b_n]$ and heights $h_n$ such that, for some
 $0<c<1/2$, for every $n \in \N$ there exists a point $x_n \in J_n$ such that
\begin{equation} \label{eq:locationzero} x_n\in [a_n+c/q_n,b_n-c/q_n], \quad \text{and} \quad S_{h_n} f'(x_n)=0,\end{equation}
then the special flow $(T_t^f)_{t\in \R}$ is ergodic and has purely singular spectrum.
\end{prop}
The proof is given below, using the following two lemmas.

\begin{lem}[Exponential tails]\label{thm:exptails}
 Suppose that $f \in \pSymLog{\sqcup_{0\leq i<d}  I_i}$.
Under the same assumptions as in Proposition~\ref{prop:an},
if $\cC'_n$ is a subtower of $\cC_n$ of height $h_n$ and whose base is $J'_n=[a_n+2\ep_n,b_n-2\ep_n]$,
then there exists $B>0$ such that for every $n\geq 1$ and $t\geq B$ we have
{\[
\lambda(\{x\in \cC'_n:|S_{h_n}(f)(x)-S_{h_n}(f)(y_n)|\geq t\})\leq |I|e^{1/c}e^{-(t-B)/2C},
\]}
where
\[  C:=\frac{\pi^2}{6}\sum_{i=0}^{d-1}C_i^+=\frac{\pi^2}{6}\sum_{i=1}^dC_i^-, \qquad y_n=\frac{a_n+b_n}{2}.\]
\end{lem}
\begin{proof}
\noindent {\it Step 1 ($x$ in the base).} Assume first that $x\in \cC'_n$ belongs to $J'_n$. 
Since by assumption there exists $x_n\in [a_n+c/q_n,b_n-c/q_n]$ such that $S_{h_n}(f')(x_n)=0$,
 we can apply Corollary~\ref{cor:exptails} which shows that
{\begin{equation}\label{eq:qla}
q_n\lambda(\{x\in J'_n:|S_{h_n}(f)(x)-S_{h_n}(f)(y_n)|\geq t\})\leq e^{1/c}e^{-t/2C}.
\end{equation}}

\smallskip
\noindent {\it Step 2 (comparing $y \in J_n'$ and $x=T^hy$ for $0\leq h<h_n$).}  Consider now any  $x\in \cC'_n$ and write it as  $x=T^hy$ for some $y\in[a_n+2\ep_n,b_n-2\ep_n]$ and $0\leq h<h_n$. Then
\[|S_{h_n}(f)(x)-S_{h_n}(f)(y)|=|S_{h}(f)(y)-S_{h}(f)(T^{h_n}y)| \]
with \[|y-T^{h_n}y|\leq\ep_n\;\text{ and }\;y,T^{h_n}y\in [a_n+\ep_n,b_n-\ep_n].\] Hence, by \eqref{eq:Sf},
{\begin{align*}
|S_h(f)(y)-S_h(f)(T^{h_n}y)|\leq 2C\big(1+\ep_nq_n(1+\log q_n)\big)\leq 2C\Big(1+\frac{1+\log q_n}{\log q_n}\Big)\leq 6C=:B.
\end{align*}}
Therefore, $|S_{h_n}(f)(x)-S_{h_n}(f)(y)|\leq B$.

\smallskip
\noindent {\it Step 3 (general case).}
By the triangle inequality, adding and subtracting $S_{h_n}(f)(y)$, where $y$ is chosen so that $x=T^hy$ as in {\it Step 2}, we have that
$|S_{h_n}(f)(x)-S_{h_n}(f)(y_n)|\geq t $ implies $|S_{h_n}(f)(y)-S_{h_n}(f)(y_n)|\geq t-B$.  In view of \eqref{eq:qla}, it follows that
{\[
\lambda(\{x\in \cC'_n:|S_{h_n}(f)(x)-S_{h_n}(f)(y_n)|\geq t\})\leq \frac{h_n}{q_n}e^{1/c}e^{-(t-B)/2C}\leq |I|e^{1/c}e^{-(t-B)/2C}.
\]}
This concludes the proof.
\end{proof}

Since any function $f\in \SymLog{\sqcup_{i=0}^{d-1}I_i}$ by definition can be written as $f=f_p+g$, where $f_p\in \pSymLog{\sqcup_{i=0}^{d-1}I_i}$ and $g$ is a function of bounded variation, the last element we need to prove Proposition~\ref{prop:an} is to control Birkhoff sums of $g$, through the following standard Denjoy-Koksma-type Lemma.

\begin{lem}[Estimate for bounded variation functions]\label{lem:BV}
Let $g:I\to\R$ be a function of bounded variation equal to $V\geq 0$. Then for all $x\in\cC'_n$ and $x'\in J'_n$
we have
\[|S_{h_n}(g)(x)-S_{h_n}(g)(x')|\leq 2V.\]
\end{lem}
\begin{proof}
First note that if $x,x'\in J_n$ and $0\leq h<h_n$ then
\[|S_{h}(g)(x)-S_{h}(g)(x')|\leq V.\]
Indeed, since
\[|S_{h}(g)(x)-S_{h}(g)(x')|\leq \sum_{0\leq j<h}|g(T^jx)-g(T^jx')|\]
and the intervals $[T^jx,T^jx']$ for $0\leq j<h$ are pairwise disjoint,
the right sum is bounded from above by the variation of $g$.

If $x\in \cC'_n$ then $x=T^hy$ for some $y\in[a_n+2\ep_n,b_n-2\ep_n]$ and $0\leq h<h_n$.
Then
\[|S_{h_n}(f)(x)-S_{h_n}(f)(y)|=|S_{h}(f)(y)-S_{h}(f)(T^{h_n}y)|\]
with $|y-T^{h_n}y|\leq\ep_n$ and $y,T^{h_n}y\in  J_n$.
It follows that
\[|S_{h_n}(f)(x)-S_{h_n}(f)(x')|\leq|S_{h}(f)(y)-S_{h}(f)(T^{h_n}y)|+|S_{h_n}(f)(y)-S_{h_n}(f)(x')|\leq 2V.\]
\end{proof}

We can now use  Lemmas~\ref{thm:exptails}~and~\ref{lem:BV} to prove Proposition~\ref{prop:an}.
\begin{proof}[Proof of Proposition~\ref{prop:an}]
We will verify the assumptions of the singularity criterion (Theorem~\ref{thm:singcrit}).
Let us remark first that, by assumption, $T$ is a rank $1$ transformation and hence it is ergodic, see \cite[Theorem~2]{Fe}.   To check the exponential tails assumption, recall first that,
by definition, $f\in \SymLog{\sqcup_{i=0}^{d-1}I_i}$ (see Def.~\ref{def:SymLog} ) can be written as $f=f_p+g$ where  $f_p\in \pSymLog{\sqcup_{i=0}^{d-1}I_i}$ and $g$ has bounded variation. Thus, if $V$ denotes the total variation of $g$, by Lemma~\ref{lem:BV},
\[
|S_{h_n}(f)(x)-S_{h_n}(f)(y_n)|\leq |S_{h_n}(f_p)(x)-S_{h_n}(f_p)(y_n)|+2 V\;\text{ for every}\; x\in \cC'_n
\]
and, $|S_{h_n}(f)(x)-S_{h_n}(f)(y_n)|\geq t$ implies $|S_{h_n}(f_p)(x)-S_{h_n}(f_p)(y_n)|\geq t-2V$.  Thus, by Lemma~\ref{thm:exptails} applied to the function $f_p$,     for every $n\geq 1$ and $t\geq B+2V$ we have
\[
\lambda(\{x\in \cC'_n:|S_{h_n}(f)(x)-S_{h_n}(f)(y_n)|\geq t\})\leq |I|e^{1/c}e^{-(t-B-2V)/2C},
\]
with $\lambda(\cC'_n)\to 1$.  Since by assumption $(h_n)$ is a rigidity sequence and the previous equation gives the exponential tails assumption \eqref{eq:expdecay}, the singularity criterion given by Theorem~\ref{thm:singcrit} implies that the special flow $(T_t^f)_{t \in \R}$ has singular spectrum.
\end{proof}

\subsection{Final arguments.}\label{sec:final}
We will now show how to conclude the proof of Theorems~\ref{thm:singsp} and~\ref{thm:ss_sf}. We will use Proposition~\ref{prop:an} (which we just proved) and Propostion~\ref{prop:cyl}, which we will prove in the next and final section.
To prove Theorem~\ref{thm:singsp}, we also need the following Lemma, which relates the notion of \emph{good rigidity} (see Definition~\ref{def:IETrecurrence}) to the conclusion of Proposition~\ref{prop:cyl}.

\begin{lem}[Good rigidity from cylinders] \label{lem:rigidity_from_cyl}
Let $(h_t)_{t\in \mathbb{R}}$ be the vertical  translation flow on an area one translation surface $(M, \omega)$. Assume that there exists a sequence  $(C_n)_{n \in \mathbb{N}}$ of cylinders with $a({C_n})\to 1  $ and $\ell({C_n})\to+\infty$ as $n\to+\infty$ such that
\begin{equation}\label{eq:app_cyl}
\vert \theta_{C_n}-\tfrac \pi 2\vert <\frac 1 {\ell({C_n})^2 \log(\ell({C_n}))}.
\end{equation}
 Let $I\subset M$ be a horizontal interval such that both of  its endpoints lie on a separatrix and are the first meeting point (forward or backward) of the separatrix and $I$.
If $T: I \to I$ is an IET obtained as the  Poincar{\'e} map of $(h_t)_{t\in \mathbb{R}}$, then $T$ admits a good rigidity sequence.
\end{lem}
The idea behind  Lemma~\ref{lem:rigidity_from_cyl} is simply that towers for the IET can be essentially obtained  intersecting the cylinders $C_n$ with the   Poincar{\'e} section. Before we prove Lemma~\ref{lem:rigidity_from_cyl}, we make the following remark which simplifies the analysis.
\begin{rem}\label{rem:corcyl}
We will without loss of generality assume that the endpoints of $I=[a,b]$ do not belong to ${C}_n$ for every $n\in\N$. Indeed, assume that $a\in C_n$. 
By definition of $I$, $a=h_{-{s}}(\sigma)$, for a singularity $\sigma\in M$ and $|{s}|<C$ (where $C>0$ is a constant independent on $n$, chosen to be an upper bound for backward and forward first return times of singularities to the section).
Since $h_{s}(a)=\sigma\notin C_n$, we find ${s'}$ between $0$ and ${s}$ such that $h_{{s}'}(a)\in\partial C_n$ and $|{s}'|$ is the smallest
positive real number with such property.
Choose $v\in \partial C_n$ so that the triangle with vertices $h_{{s}'}(a)$,  $v$, $a$ is right ($v$ is its right angle vertex) and contained in $\overline{C_n}$. Then $\tfrac{d(a,v)}{|{s}'|}=|\sin(\pi/2-\theta_{C_n})|$.
It follows that
\[
d(a,\partial C_n)\leq d(a,v)=|{s}'||\sin(\pi/2-\theta_{C_n})|\leq |{s}||\pi/2-\theta_{C_n}|\leq  \frac{C}{\ell(C_n)^2\log(\ell(C_n))}.
\]
 Analogous estimates hold for the other endpoint. If we define \emph{trimmed} cylinders $(C'_n)_{n\in\mathbb{N}}$ given by:
\[
C'_n:=C_n\setminus \Big\{x\in C_n\;:\; d(x,\partial C_n)\leq \frac{C}{\ell(C_n)^2\log(\ell(C_n))}\Big\},
\]
then the discarded set has measure at most $\frac{2C}{\ell(C_n)\log(\ell(C_n))}$ and therefore we also have that $a(C'_n)\to 1$ as $n$ grows. Since $\theta_{C'_n}=\theta_{C_n}$ and $\ell(C'_n)=\ell(C_n)$, the sequence of cylinders $(C'_n)_{n \in \mathbb{N}}$ also satisfies  \eqref{eq:app_cyl}.  Therefore we can replace the sequence $(C_n)_{n\in\mathbb{N}}$ with the sequence of cylinders $(C'_n)_{n\in\mathbb{N}}$,  which by construction do not contain the endpoints of $I$.
\end{rem}
With the above remark we can now prove Lemma~\ref{lem:rigidity_from_cyl}.
\begin{proof}[Proof of Lemma~\ref{lem:rigidity_from_cyl}]
Let $(C_n)_{n \in \mathbb{N}}$ be the sequence of cylinders satisfying the assumption of  Lemma~\ref{lem:rigidity_from_cyl} and let $I$ denote, by abusing the notation, also the horizontal interval on $M$ which gives the Poincar{\'e} section determining $T$. Let $(\overline{a}_n,\overline{b}_n)\subset I$ be one of  connected components  of the intersection of $C_n$ with $I$ 
and set $h_n$ to be the number of connected components in $I \cap C_n$.
In view of Remark~\ref{rem:corcyl}, all these connected components are horizontal intervals of length $\overline{b}_n-\overline{a}_n$
 whose endpoints both lie on $\partial C_n$.
The images of the intervals $(\overline{a}_n,\overline{b}_n)$ under the vertical flow (and hence the successive intersections of  $C_n$ with $I$) are not necessarily disjoint, but since the vertical flow is close  by \eqref{eq:app_cyl} to the direction $\theta_{C_n}$ of the cylinder, to obtain the base of a tower it is sufficient to \emph{trim} the interval as follows. Let  $J_n:=({a}_n,{b}_n)$ to be the smaller interval given  by
\[a_n:= \overline{a}_n+ \epsilon_n, \qquad b_n:= \overline{b}_n- \epsilon_n ,\qquad \text{where} \ \ \epsilon_n:= \frac{1}{\ell({C_n}) \log(\ell({C_n}))}. \]
Then, by \eqref{eq:app_cyl}  and elementary trigonometry (see Figure \ref{fig:cylinder}),  we have that the symmetric difference of $T^{h_n}J_n$ and $J_n$ has length
\[
\ell({C_n}) \big|\sin \left(  \theta_{C_n}-\tfrac \pi 2 \right)\big|\leq \ell({C_n}) \vert \theta_{C_n}-\tfrac \pi 2\vert \leq \frac{1}{\ell({C_n})  \log(\ell({C_n}))} =\epsilon_n.
\]
Note first that by Remark~\ref{rem:corcyl} it follows that the sets $T^h(J_n)$ for $0\leq h <h_n$ are intervals. Moreover, they are pairwise disjoint and we also have that  $\vert T^{h_n}(x) - x \vert \leq \epsilon_n$ for any $x$ from these intervals. Setting
$(\cC_n)_{n \in \mathbb{N}}$ to be given by $\cC_n:=\cup_{h=0}^{h_n-1}T^h (J_n)$, we obtain towers which satisfy $(GR2)$ from the Definition~\ref{def:IETrecurrence}. To finish the proof it is enough to show that the towers $(\cC_n)_{n \in \mathbb{N}}$ satisfy $(GR1)$.  Consider the subsurface $F_n\subset S$ (obtained flowing $J_n$) given by  $F_n:=\bigcup_{0\leq t<\sin(\theta_{C_n})\ell(C_n)}h_t(J_n)$ (the flowing time $\sin(\theta_{C_n})\ell(C_n)$ is here the smallest first return time of points in $J_n$ to $J_n$, cf.~Figure~\ref{saddles2}). Denote by $p_I:M\to I$ the projection along the vertical flow of $M$ on $I$ defined by setting $p_I(x)$ to be the first
meeting point of the backward orbit of $x$ under $(h_t)_{t\in\R}$ with $I$.
Then  $\cC_n=p_I(F_n)$. Moreover, by the bound on $\theta_{C_n}$,
\begin{align*}
a(F_n)&=|J_n|\cdot \ell(C_n)\sin(\theta_{C_n})=(\bar{b}_n-\bar{a}_n-2\epsilon_n)\ell(C_n)\sin(\theta_{C_n})\\
&= (\bar{b}_n-\bar{a}_n)\ell(C_n)\sin(\theta_{C_n})- 2\epsilon_n\ell(C_n)\sin(\theta_{C_n})\geq a(C_n)-\frac{2}{\log(\ell(C_n))}\to 1.
\end{align*}
Then, if  $c>0$ denotes the minimum of all backward first return times of points from $I$ to $I $ and recall that $|\cdot|$ denotes the Lebesgue measure of a (measurable) subset of $I$, we have
\[
c\left|p_I(M\setminus F_n)\right|\leq a(M\setminus F_n)\to 0.
\]
 Therefore $\left|p_I(M\setminus F_n)\right|\to 0$. Consequently
\[
|\cC_n|=|p_I(F_n)|=|I|-\left|p_I(M\setminus F_n)\right|\to |I|.
\]
This finishes the proof of $(GR1)$.
\end{proof}

 \begin{figure}[h!]
  \subfigure[\label{fig:cylinder} The cylinder in the proof of Lemma \ref{lem:rigidity_from_cyl}]{
  \includegraphics[width=0.5\textwidth]{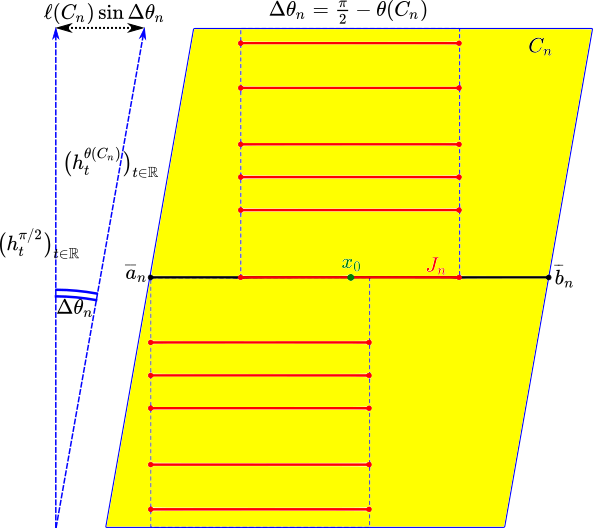}	}
\subfigure[$x=T^hy$ in Step 2 of the proof of Proposition~\ref{prop:an}\label{fig:tower}]{
\includegraphics[width=0.4\textwidth]{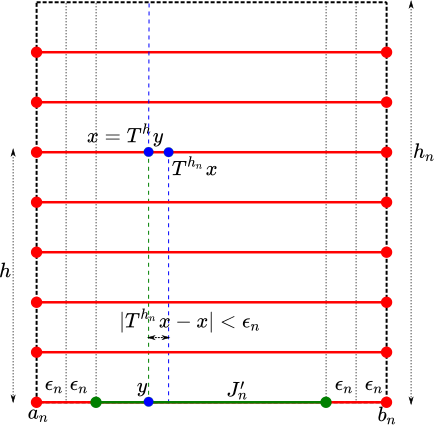}}	
 \caption{Auxiliary figures for the proofs of Lemma \ref{lem:rigidity_from_cyl} and Proposition~\ref{prop:an}.\label{saddles2}}
\end{figure}

We can now prove Theorem~\ref{thm:ss_sf}.  Let us first outline the strategy of the proof.
Singularity of the spectrum for special flows satisfying the assumptions of Theorem~\ref{thm:ss_sf} will be deduced  from an application of Proposition~\ref{prop:an}. Hence  we only need to verify that the assumptions (good rigidity and location of zeros of the derivatives) hold for almost every $T$ with permutation $\pi$.  In  {\it Part 1}, \emph{good rigidity} is deduced from Proposition~\ref{prop:cyl} via Lemma~\ref{lem:rigidity_from_cyl}.  In {\it Part 2}, the assumption \eqref{eq:locationzero} on the \emph{location of zeros} of the derivative is proven exploiting the symmetry of the function and the cancellation phenomena described in \S\ref{sec:symmetry} (via Corollary~\ref{cor:x0}).

\begin{proof}[Proof of Theorem~\ref{thm:ss_sf}.]

\noindent {\it Part 1.}  Let $d\geq 2$ and consider the symmetric permutation $\pi$ on $d$ symbols.
Consider translation surfaces (for example obtained by choosing suspension data, see \cite{Yo}) which has an IET $T$, with permutation $\pi$ as a Poincar\'{e} section. 
  Any such translation surface $(M, \omega)$ belongs to the  stratum $\mathcal{H}=\mathcal{H}(2g-2)$ where $g=d/2$ if $d$ is even, or to $\mathcal{H}=\mathcal{H}(g-1,g-1)$
where $g=(d-1)/2$ if $d$ is odd.
By Proposition~\ref{prop:cyl},  for almost every translation surface in (any connected component of)  the stratum $\mathcal{H}(2g-2)$ or $\mathcal{H}(g-1,g-1)$,
 the vertical flow is well approximated by single cylinders in the sense of Proposition~\ref{prop:cyl}.
  Furthermore, since Proposition~\ref{prop:cyl} holds for every $\epsilon>0$, taking a sequence $\epsilon_n\to 0$ and using a diagonal argument, we also have that  for a full measure set $\mathscr{F}_\mathcal{H}$ of  translation surface  in either strata there exists a sequence of cylinders $(C_n)_{n\in \mathbb{N}}$ with  $a(C_n)\to 1$ and satisfying \eqref{eq:app_cyl}.  We can also assume that for every surface $(M,\omega)$ in $\mathscr{F}_\mathcal{H}$ the corresponding horizontal flow has no saddle connection.

By standard arguments (using Fubini theorem and the local product structure of the Masur-Veech measure on translation surfaces), we hence get a  full measure set $\mathscr{F}_{\pi}$ of IETs with permutation $\pi$ (those which arise as Poincar{\'e} sections of the vertical flow on surfaces in $\mathscr{F}_\mathcal{H}$) that, by Lemma~\ref{lem:rigidity_from_cyl}, 
admit a good rigidity sequence (in the sense of Definition~\ref{def:IETrecurrence}).


\smallskip
 \noindent  {\it Part 2.}
Given $T\in \mathscr{F}_{\pi}$, let $(M, \omega)$ be a translation surface in $ \mathscr{F}_\mathcal{H}$ of which $I$ is a horizontal section.
Let $\iota: M \to M$ denote the hyperelliptic involution on $(M, \omega)$  (see \S\ref{sec:reduction}).
{ Since $\pi$ is symmetric,} the midpoint $x_0=|I|/2$ of $T$ is fixed by $\iota$ (i.e.\ is a Weierstrass point). Let  $(C_n)_{n\in \mathbb{N}}$ be  the sequence of cylinders on $(M,\omega)$  satisfying \eqref{eq:app_cyl} and $a({C_n})\to 1  $.  We can assume that each cylinder $C_n$ is maximal. Since, for  $n$ sufficiently large, $a({C_n})>1/2$,  we must have $\iota(C_n) = C_n$. { Indeed, since $\iota$ maps cylinders into cylinders and preserves area, $\iota(C_n)$ is a cylinder intersecting $C_n$ (as $a(\iota(C_n))>1/2$ and $a(C_n)>1/2$).
Therefore $C_n\cup \iota(C_n)$ is also a cylinder. By the maximality of $C_n$, it follows that  $\iota(C_n) = C_n$.}
This implies that there is a Weierstrass point (actually exactly two) on the core curve of $C_n$. We claim that, without loss of generality, we can assume that for all $n \in \N$ this Weierstrass point is the mid-point $x_0$ of $I$. (Indeed, if it is not, we can replace the section $I$ with another symmetric section centered at the given Weiestrass point; then { by Remarks~\ref{rem:invertcond} and \ref{rem:symcond} (applied to $f=f_p$)} this new section yields a special flow which also satisfies the assumption of Theorem~\ref{thm:ss_sf} and both special flows have the same spectral properties since they are both metrically isomorphic to the same surface flow.)

We now claim that the sequence of good rigidity towers given by Lemma~\ref{lem:rigidity_from_cyl} can be choosen so that
\begin{equation}\label{goodx0andreturn}
x_0 \in [a_n+2c/q_n,b_n-2c/q_n], \qquad  T^{-h_n}x_0\in [a_n+c/q_n,b_n-c/q_n]\;\text{ for some}\;0<c<1/4.
\end{equation}
Indeed, since the midpoint  $x_0$ of $I$ belongs to the core curve of $C_n$ for every $n$, { we can choose $J_n=[{a}_n,{b}_n]\subset I$
such that $({a}_n-\ep_n,{b}_n+\ep_n)$ is the unique connected component of the intersection $I$ with $C_n$ that contains $x_0$
(see the proof of Lemma~\ref{lem:rigidity_from_cyl}). Then $x_0=({a}_n + {b}_n)/2$.} Fix $0<c<1/4$ so that, for every $n$ sufficiently large, $\ep_n\leq c/q_n$. Then, $x_0\in[a_n+2c/q_n,b_n-2c/q_n] \subset [a_n,b_n] $.
Consider now $T^{-h_n}x_0$. Since  $|T^{-h_n}x_0-x_0|\leq \ep_n\leq c/q_n$
and $x_0\in [a_n+2c/q_n,b_n-2c/q_n]$, we have that  also $T^{-h_n}x_0\in [a_n+c/q_n,b_n-c/q_n]$, which concludes the proof of \eqref{goodx0andreturn}. 

From  \eqref{goodx0andreturn}, since $x_0, T^{-h_n}x_0\in [a_n+c/q_n,b_n-c/q_n]\subset (a_n,b_n)$ which by assumption is a continuity interval for $S_{h_n}(f)$ (see the remark after the Definition~\ref{def:tower}),  
we deduce by Corollary \ref{cor:x0} that there exists  $x_n\in [a_n+c/q_n,b_n-c/q_n]$ such that $S_{h_n}(f')(x_n)=0$.
This shows that all assumptions of Proposition~\ref{prop:an} hold for the special flow $(T^f_t)_{t \in \R}$  over $T\in \mathscr{F}_{\pi} $ and hence (by Proposition~\ref{prop:an}), that $(T^f_t)_{t \in \R}$ has purely singular spectrum.
\end{proof}

\begin{proof}[Proof of Theorem~\ref{thm:singsp}.]
Any $(\varphi_t)_{t\in \mathbb{R}}$ locally Hamiltonian flow with two simple isomorphic saddles on $M$ of genus two, by Corollary \ref{cor:reduction}, is metrically isomorphic to a special flow over $T$ over a symmtric $\pi$ (given by $\pi(i)=4-i$, $0\leq i <5$) under $f\in  \SymLog{\sqcup_{i=0}^{4}I_i}$ (and hence has the same ergodic and spectral properties). By Theorem~\ref{thm:ss_sf}, for almost every choice of the lengths $\beta_{i+1}-\beta_i$ of $T$, such special flow has purely singular spectrum. This, by Remark~\ref{rem:fullmeasure}, implies singularity of the spectrum for a full measure set of locally Hamiltonian flows in in the isomorphic saddles locus $\mathcal{K}$ with respect to the Katok fundamental class (defined in \S\ref{sec:isomorphic}).
\end{proof}

\section{Translation surfaces well approximated by single cylinders}\label{sec:cylinder}

In this section we will prove the results on the abundance of single cylinders in translation surfaces stated in \S\ref{sec:main_cyl}, namely Theorem~\ref{thm:Kcyl} and Proposition~\ref{prop:cyl}. Let us first show how Proposition~\ref{prop:cyl} follows from Theorem~\ref{thm:Kcyl}.

\begin{proof}[Proof of Proposition~\ref{prop:cyl}] Let $\psi(t)=\frac{1}{t^2\log t}$. Then $\psi$ satisfies the assumptions of Theorem~\ref{thm:Kcyl}.  Moreover let  $\epsilon_n=\frac{1}{n}$. Notice that a.e.\ $(M,\omega)$ belongs to the intersection of the full measure sets coming from Theorem~\ref{thm:Kcyl} (intersection over the $(\epsilon_n)_{n\in \N}$). It remains to notice that every such $(M,\omega)$ satisfies the assertion of Proposition~\ref{prop:cyl}.
\end{proof}

The rest of this section is devoted to the proof of Theorem~\ref{thm:Kcyl}. From now on we will constantly assume that $0<\epsilon <1/2$. For every $\theta\in S^1$ and $r>0$ let $B(\theta,r)=\{\phi\in S^1:\|\phi-\theta\|<r\}$. Let $\mathcal{C}$ be a connected component in the moduli space of area one translation surfaces. Theorem~\ref{thm:Kcyl} will be a consequence of the following result.

\begin{prop} \label{thm:approx}
{Let $\psi:\mathbb{R}^+ \to \mathbb{R}^+$ be non-increasing so that $t\psi(t)\leq 1$ for $t$ large enough and $\int_1^{+\infty} t \psi(t)=\infty$.}
 For every $0<\epsilon<1/2$ there exists $0<c\leq 1$ such that a.e.\ $(M,\omega)\in\mathcal{C}$ and every interval $J\subset S^1$  satisfies
\begin{equation}\label{eq:BC}
cT^2 \lambda(J)<\#\{C\in {\mathcal{C}yl}^\epsilon_\omega:\ell(C)\leq T\text{ and }\theta_C \in J\} \text{ for all  }T\geq T_{\omega,J},
\end{equation}
for some $T_{\omega,J}>0$. Moreover, if \eqref{eq:BC} holds, then for $T\geq \max(T_{\omega, J}, 36/(c\lambda(J)))$, we have
\[\lambda\Big(\bigcup_{\{C\in {\mathcal{C}yl}^\epsilon_\omega:\ell(C)\geq T\}}B(\theta_C, \psi(\ell(C)))\cap J\Big)\geq\frac c 9\lambda(J).\]
\end{prop}
The above result can be proved by a modification of the methods of \cite{Ch} and \cite{Ma-Tr-We}. We will present a full proof for completeness in \S\ref{sub:new}.

\smallskip
 Let us now  show how it implies Theorem~\ref{thm:Kcyl}.

\begin{proof}[Proof of Theorem~\ref{thm:Kcyl}]
Let
\[W^\psi_{\omega,m}:=\bigcup_{\{C\in {\mathcal{C}yl}^\epsilon_\omega:\ell(C)\geq m\}}B(\theta_C,\psi(\ell(C))).\]
Then the sequence of sets $(W^\psi_{\omega,m})_{m\geq 1}$ is  non-increasing with $\bigcap_{m\geq 1}W^\psi_{\omega,m}=W^\psi_{\omega}$. To proof
\eqref{eq:set} we need to show
that for a.e.\ $(M,\omega)\in \mathcal{C}$ and every $m\geq 1$ the set $W^\psi_{\omega,m}\subset S^1$ has full measure.

In view of Proposition~\ref{thm:approx}, for a.e.\ $(M,\omega)\in \mathcal{C}$ and any interval $J\subset S^1$
we have
\begin{equation}\label{eq:low}
\frac{\lambda\big(W^\psi_{\omega,m}\cap J\big)}{\lambda(J)}\geq\frac{c}{9}\text{ if }m\geq T_{\omega,J}\text{ and }{m\geq 36/(c\lambda(J))}.
\end{equation}
Take any translation surface $(M,\omega)$ satisfying  the above condition and suppose, contrary to our claim, that for some $m_0\geq 1$
the set $W^\psi_{\omega,m_0}\subset S^1$ does not have full measure. By the Lebesgue density theorem there exists an interval
$J\subset S^1$ such that
\[\frac{\lambda\big(W^\psi_{\omega,m}\cap J\big)}{\lambda(J)}\leq \frac{\lambda\big(W^\psi_{\omega,m_0}\cap J\big)}{\lambda(J)}<\frac{c}{9}
\text{ for all }m\geq m_0,\]
contrary to \eqref{eq:low}. This gives \eqref{eq:set}.

Denote by $A\subset \mathcal{C}$ the set of translation sufaces $(M,\omega)\in \mathcal{C}$ for which there exists a sequence $(C_i)_{i\geq 1}$ in
${\mathcal{C}yl}_\omega^\epsilon$ such that $\ell(C_i)\to+\infty$ as $i\to+\infty$ and $\|\theta_{C_i}-\frac{\pi}{2}\|<\psi(\ell(C_i))$ for all $i\geq 1$.
In view of \eqref{eq:set} there exists a subset $A'\subset \mathcal{C}$ with $\nu_C(A')=1$ such that if $(M,\omega)\in A'$ then for every $\phi\in W^{\psi}_\omega$ we have
\[r_{\pi/2-\phi}\omega\in A\quad\text{and}\quad\lambda(W^{\psi}_\omega)=1.\]
Let us consider the continuous map
\[\Delta:S^1\times \mathcal{C}\to \mathcal{C},\quad \Delta(\theta,\omega)=r_{\pi/2-\theta}\omega.\]
By Fubini's theorem and the invariance of $\nu_{\mathcal C}$ under the action of rotations $(r_\theta)_{\theta\in S^1}$, we have $\Delta_*(\lambda\times\nu_{\mathcal C})=\nu_{\mathcal C}$. Moreover,
\[\bigcup_{\omega\in A'}(W_{\omega}^\psi\times\{\omega\})\subset\Delta^{-1}(A).\]
Using again Fubini's theorem, we obtain
\[1=(\lambda\times\nu_{\mathcal C})(\Delta^{-1}(A))=\nu_{\mathcal C}(A),\]
which completes the proof.
\end{proof}

\subsection{Proof of Proposition \ref{thm:approx}}\label{sub:new}
The first part of Proposition \ref{thm:approx} (i.e.~\eqref{eq:BC}) is an immediate consequence of the following result, which follows by Theorem~1.9 in \cite{Vor}:
\begin{thm}
For a.e.\ translation surface $(M,\omega)\in\mathcal{C}$ and all intervals $J\subset S^1$, $I\subset [0,1]$ we have
\[\lim_{T\to+\infty}\frac{\#\{C\in{\mathcal{C}yl}_\omega: \ell(C)\leq T,\theta_C\in J,a(C)\in I\}}{T^2}=c_1(\mathcal{C})\lambda(J)|I|^{m_{\mathcal{C}}-1}.\]
\end{thm}
So it remains to prove the second part of Proposition \ref{thm:approx}.
For this we first state some additional lemmas.




\begin{lem}\label{lem:flow}
Let $(M,\omega)$ be any translation surface.
If $x\in M$ belongs to two different cylinders $C, C'\in {\mathcal{C}yl}_\omega$ then $\|\theta_{C}-\theta_{C'}\|\geq \frac{\max\{a(C),a(C')\}}{\ell(C)\ell(C')}$.
\end{lem}
\begin{proof}
For every $\theta\in S^1$ denote by $(h^\theta_t)_{t\in\R}$ the directional translation flow on $(M,\omega)$ in direction $\theta$.
Notice that the  circumference of the cylinder $C$ is  $\frac{a(C)}{\ell(C)}$. Suppose that $x\in C$ is a periodic point for
$(h^\phi_t)_{t\in\R}$ for some $\phi\neq \theta_C$ and
$R>0$ is its minimal period, i.e.\ $h_s^{\phi}(x)=h_{s-R}^{\phi}(x)$ for all $s\in \R$. Choose $s\in \R$ so that $h_s^{\phi}(x)$ is just leaving the periodic cylinder $C$.
So $h^{\phi}_{s-t}(x)\in C$  for all $0<t<\frac{a(C)}{\ell(C)}|\csc(\theta_C-\phi)|$ and in particular
$h_{s-t}^{\phi}(x)\neq h^{\phi}_s(x)$.
Therefore  $R\geq \frac{a(C)}{\ell(C)}|\csc (\theta_C-\phi)|$ which implies that $\|\theta_C-\phi\| \geq \frac{a(C)}{R\ell(C)}$.
\end{proof}

\begin{cor}\label{cor:sep}
If $0<\epsilon <1/2$ then the set
\[\{\theta_C\in S^1:C \in{\mathcal{C}yl}^\epsilon_\omega\text{ such that }\ell(C)\leq T\}\]
is $\frac{1-\epsilon}{T^2}$ separated. In particular,  for any interval $J$ and $T>0$ we have
{\[\#\{C\in {\mathcal{C}yl}^\epsilon_\omega:\ell(C)\leq T\text{ and }\theta_C \in J\}\leq  2T^2 \lambda(J)+1.\]}
\end{cor}
\begin{proof}
Since the cylinders have area greater than $\frac 1 2 $, any pair of cylinders must share a point. The statement then follows from Lemma \ref{lem:flow}.
\end{proof}

For any interval $J=B(\theta,r)\subset S^1$ and
any $s>0$ let $J^{+s}:=B(\theta,r+s)$.

\begin{lem} \label{lem:add more}
Let $\sigma:=18\sqrt{c^{-1}}>1$ and $0<\epsilon < 1/2$. Assume that $J\subset S^1$ is an interval and $T\geq 36/(c\lambda(J))$  satisfy \eqref{eq:BC} and
\begin{equation}\label{eq:assum}
\lambda\Big(\bigcup_{\{C\in {\mathcal{C}yl}^\epsilon_\omega:L\geq \ell(C)\geq T\}}B(\theta_C, \psi(\ell(C)))\cap J\Big)<\frac {c} 9\lambda(J)
\end{equation}
for some  $L>T$. Then
\begin{align*}
\lambda\Big(\Big(&\bigcup_{\{C\in {\mathcal{C}yl}^\epsilon_\omega:\sigma L\geq \ell(C)\geq L\}}B(\theta_C, \psi(\ell(C)))\setminus \bigcup_{\{C\in {\mathcal{C}yl}^\epsilon_\omega:L\geq \ell(C)\geq T\}}B(\theta_C, \psi(\ell(C)))\Big)\cap J\Big)\\
&>\min\{{(\sigma L)^2}\psi(\sigma L),1 \}\frac{c}{4}\lambda(J).
\end{align*}
\end{lem}
\begin{proof}
As $\sigma L>L>T\geq T_{\omega,J}$, by \eqref{eq:BC} and Corollary~\ref{cor:sep}, the set
\[\Theta:=J\cap\bigcup_{\{C\in {\mathcal{C}yl}^\epsilon_{\omega}:\ell(C)\leq \sigma L\}}\{\theta_C\}\subset S^1\]
has at least $c(\sigma L)^2 \lambda(J)$
points that are $\frac 1 {2(\sigma L)^2}$ separated. Denote by $\partial \Theta$ the set of two points in $\Theta $
which {are the closest to the ends} of the interval $J$. Then for every $\theta\in\Theta\setminus\partial\Theta$ we have
\begin{equation}\label{eq:BJ}
B\big(\theta,  \min\{\psi(\sigma L),(\sigma L)^{-2}/2\}\big)\subset J.
\end{equation}

Since $\psi(T)\leq 1/T\leq c\lambda(J)/36$,
by Corollary~\ref{cor:sep}, we have
\begin{equation}\label{eq:+psi}
\#\{C\in {\mathcal{C}yl}^\epsilon_\omega:L\geq \ell(C)\geq T,\;\theta_C \in J^{+\psi(T)}\} \leq 2\lambda(J^{+\psi(T)})L^2+1\leq 4\lambda(J)L^2+1.
\end{equation}
Moreover
\begin{equation}\label{eq:subset+}
\bigcup_{\{C\in {\mathcal{C}yl}^\epsilon_\omega:L\geq \ell(C)\geq T\}}B(\theta_C, \psi(\ell(C)))\cap J
\subset \bigcup_{\{C\in {\mathcal{C}yl}^\epsilon_\omega:L\geq \ell(C)\geq T,\;\theta_C \in J^{+\psi(T)}\}}B(\theta_C,\psi(\ell(C)))
\end{equation}
and
\begin{align}\label{eq:2c9}
\begin{aligned}
\lambda&\Big(\bigcup_{\{C\in {\mathcal{C}yl}^\epsilon_\omega:L\geq \ell(C)\geq T,\;\theta_C \in J^{+\psi(T)}\}}B(\theta_C,\psi(\ell(C)))\Big)\\
&\leq \lambda\Big(\bigcup_{\{C\in {\mathcal{C}yl}^\epsilon_\omega:L\geq \ell(C)\geq T\}}B(\theta_C, \psi(\ell(C)))\cap J\Big)+4\psi(T)<\frac{2c}{9}\lambda(J).
\end{aligned}
\end{align}

In view of \eqref{eq:2c9} and \eqref{eq:+psi}, the cardinality of the set $\Theta^*\subset \Theta$ of points $\theta \in \Theta$ such that
\[B\big(\theta,  \min\{\psi(\sigma L),(\sigma L)^{-2}/2\}\big)\cap\bigcup_{\{C\in {\mathcal{C}yl}^\epsilon_\omega:L\geq \ell(C)\geq T,\;\theta_C \in J^{+\psi(T)}\}}
B(\theta_C,\psi(\ell(C)))\neq \emptyset\] is at most
\begin{equation*}
3(4L^2\lambda(J)+1)+\frac {4c} 9 \lambda(J) {(\sigma L)^2}.
\end{equation*}
{Indeed, the union of $N$ intervals with total measure $\mu$ meets at most $\mu/\epsilon + N$ points which are $\epsilon$-separated.
As elements of $\Theta$ are $1/(2(\sigma L)^2)$-separated, it follows that there are at most
\[\frac{\frac{2c}{9}\lambda(J)}{\frac{1}{2(\sigma L)^2}}+4L^2\lambda(J)+1=\frac {4c} 9 \lambda(J) {(\sigma L)^2}+4L^2\lambda(J)+1\]
elements of $\Theta^*$ such that
\begin{equation}\label{eq:notin*}
\theta\in \bigcup_{\{C\in {\mathcal{C}yl}^\epsilon_\omega:L\geq \ell(C)\geq T,\;\theta_C \in J^{+\psi(T)}\}}
B(\theta_C,\psi(\ell(C))).
\end{equation}
Suppose that $\theta\in \Theta^*$ does not meet \eqref{eq:notin*}. Then $\theta$ is in the $1/(2(\sigma L)^2)$-neighbourhood
of an interval $B(\theta_C,\psi(\ell(C)))$ but it does not belong to $B(\theta_C,\psi(\ell(C)))$. As elements of $\Theta$ are $1/(2(\sigma L)^2)$-separated, for every $C$ there are at most two elements of $\Theta^*$ which do not meet \eqref{eq:notin*}. Thus
\[\#\Theta^*\leq \frac {4c} 9 \lambda(J) {(\sigma L)^2}+4L^2\lambda(J)+1+2(4L^2\lambda(J)+1)=3(4L^2\lambda(J)+1)+\frac {4c} 9 \lambda(J) {(\sigma L)^2}.\]
Since $\#\Theta\geq c(\sigma L)^2 \lambda(J)$, $\sigma=18/\sqrt{c}$ and $L>T\geq 36/(c\lambda(J))$, this gives}
\begin{align}\label{eq:intersected points}
\begin{aligned}
\#(\Theta\setminus(\partial\Theta\cup\Theta^*))&\geq c(\sigma L)^2\lambda(J)- 12L^2\lambda(J)-5-\frac {4c} 9 \lambda(J) (\sigma L)^2\\
&\geq c\frac 5 9 (\sigma L)^2 \lambda(J)-17L^2\lambda(J)>\frac 1 2 c( \sigma L)^2 \lambda(J).
\end{aligned}
\end{align}
 As $\psi$ is non-increasing, by \eqref{eq:BJ} and the definition of $\Theta^*$, for every $\theta\in \Theta\setminus(\partial\Theta\cup \Theta^*)$ we have
$B(\theta,  \min\{\psi(\sigma L),(\sigma L)^{-2}/2\})$ is a subset of
\begin{align*}
&\Big(\!\bigcup_{\{C\in {\mathcal{C}yl}^\epsilon_\omega:\sigma L\geq \ell(C)\geq L\}}B(\theta_C, \psi(\ell(C)))\!\setminus \!\bigcup_{\{C\in {\mathcal{C}yl}^\epsilon_\omega:L\geq \ell(C)\geq T,\theta_C\in J^{+\psi(T)}\}}B(\theta_C, \psi(\ell(C)))\!\Big)\cap J\\
&\subset \Big(\bigcup_{\{C\in {\mathcal{C}yl}^\epsilon_\omega:\sigma L\geq \ell(C)\geq L\}}B(\theta_C, \psi(\ell(C)))\setminus \bigcup_{\{C\in {\mathcal{C}yl}^\epsilon_\omega:L\geq \ell(C)\geq T\}}B(\theta_C, \psi(\ell(C)))\Big)\cap J,
\end{align*}
where the last inclusion follows from \eqref{eq:subset+}.
Since the centers of intervals are $(\sigma L)^{-2}/2$ separated,  by \eqref{eq:intersected points}, the measure of the last set  is at least
\[\frac 1 4 c( \sigma L)^2 \lambda(J)\min\{\psi(\sigma L),(\sigma L)^{-2}\},\]
%
which completes the proof.
\end{proof}

\begin{lem}\label{lem:diverg}
If $\psi:\mathbb{R}^+\to \mathbb{R}^+$ is bounded, non-increasing and $\int_{1}^\infty t\psi(t)=+\infty$ then for any $\sigma>1$ we have
\[\sum_{k=0}^\infty \sigma^{2k}\psi(\sigma^k)=+\infty.\]
\end{lem}
\begin{proof}
Lemma follows directly from the following
\begin{align*}
\int_{1}^\infty t\psi(t)dt=\sum_{k=0}^\infty\int_{\sigma^k}^{\sigma^{k+1}}t\psi(t)dt\leq \sum_{k=0}^\infty (\sigma^{k+1}-\sigma^k)\sigma^{k+1}\psi(\sigma^k)=
\sum_{k=0}^\infty  \sigma(\sigma-1)\sigma^{2k}\psi(\sigma^k).
\end{align*}
\end{proof}
\begin{proof}[Proof of Proposition \ref{thm:approx}]
Recall that we only need to show the second part (we already know that \eqref{eq:BC} holds).

To prove the result we need {to} show that for every interval $J\subset S^1$ there exists $k\geq \log_\sigma T$ such that
\begin{equation}\label{eq:assum1}
\lambda\Big(\bigcup_{\{C\in {\mathcal{C}yl}^\epsilon_\omega:\sigma^k\geq \ell(C)\geq T\}}B(\theta_C, \psi(\ell(C)))\cap J\Big)\geq\frac {c} 9\lambda(J).
\end{equation}
Suppose, contrary to our claim, that for all $k\geq \log_\sigma T$ \eqref{eq:assum1} does not hold.
By Lemmas \ref{lem:add more} and \ref{lem:diverg},  we have
\begin{align*}
&\lambda\Big(\bigcup_{\{C\in {\mathcal{C}yl}^\epsilon_\omega:\ell(C)\geq T\}}B(\theta_C, \psi(\ell(C)))\cap J\Big)\\
&\geq\sum_{k\geq\log_{\sigma}(T)}
\lambda\Big(\Big(\bigcup_{\{C\in {\mathcal{C}yl}^\epsilon_\omega:\sigma^{k+1}\geq \ell(C)\geq \sigma^k\}}B(\theta_C, \psi(\ell(C)))\setminus \bigcup_{\{C\in {\mathcal{C}yl}^\epsilon_\omega:\sigma^k\geq \ell(C)\geq T\}}B(\theta_C, \psi(\ell(C)))\Big)\cap J\Big)\\
&\geq \sum_{k\geq\log_{\sigma}(T)}\min\{{\sigma^{2k}}\psi(\sigma^k),1 \}\frac{c}{4}\lambda(J)=+\infty,
\end{align*}
which is a contradiction.
\end{proof}

\subsection*{Acknowledgements}
The authors would like to thank the Banach Center program in B\k{e}dlewo in summer 2018 {Conference ``Ergodic aspects of modern dynamics'' in honour of Mariusz Lema\'nczyk on his 60th birthday, B\k{e}dlewo, Poland, 10-16 June, 2018}, which played a key role to foster the authors collaboration and bring this project to successful conclusion.
For their hospitality  at the various initial stages of this work, they would also like to thanks the School of Mathematics of the University of Bristol and the Mathematics Institute of the University of Z{\"u}rich.  C.~U.\ acknowledges the ERC Grant {\it ChaParDyn} which supported her and several collaboration visits.
J.~C.\ thanks NSF grant DMS-1452762 and a Warnock chair.
A.~K.\ was partially supported by the NSF grant DMS-1956310.
The research leading to these results has received funding from the European Research Council under the European Union
Seventh Framework Programme (FP/2007-2013) / ERC Grant Agreement n.~335989. Research partially supported by the Narodowe Centrum Nauki Grant 2019/33/B/ST1/00364.

{The authors are also indebted to the anonymous referee, for his/her careful job and corrections on the first draft of this paper.}

\end{document}